\documentclass[11pt]{article}

\usepackage{amsmath,amsfonts,amssymb,amsthm,graphicx,color,mathtools}
\usepackage[normalem]{ulem} % to strikeout words

\usepackage{times, color}
\usepackage{graphicx}
\graphicspath{{./figures/}}	
\usepackage[utf8]{inputenc}
\usepackage{paralist}
\usepackage{booktabs}
\usepackage[colorlinks,pdfdisplaydoctitle]{hyperref}

\DeclareMathAlphabet{\mathbf}{OT1}{cmr}{bx}{it}

\newcommand{\supp}{\mathrm{supp}\,}
\newcommand{\tr}{\mathrm{tr}\,}
\newcommand{\rg}{\mathrm{rg}\,}
\newcommand{\mc}{\mathcal}

\DeclareMathAlphabet{\mathbf}{OT1}{cmr}{bx}{it}

\newcommand{\argmin}{\operatornamewithlimits{argmin}}

\newcommand{\essinf}{\operatornamewithlimits{ess\ inf}}
\newcommand{\esssup}{\operatornamewithlimits{ess\ sup}}

\newcommand{\bbN}{\mathbb{N}}
\newcommand{\bbR}{\mathbb{R}}

\DeclareMathOperator{\e}{\mathrm e}

\renewcommand{\d}{\mathrm{d}}

 \theoremstyle{definition}
 \newtheorem{defi}{Definition}%[chapter]
 \newtheorem{exam}[defi]{Example}
 \newtheorem{rem}[defi]{Remark}
 \newtheorem{propo}[defi]{Proposition}
 \newtheorem{theo}[defi]{Theorem}
 \newtheorem{lem}[defi]{Lemma}
 \newtheorem{cor}[defi]{Corollary}
\newcommand{\W}{\mathrm{W}}
\newcommand{\Lip}{\mathrm{Lip}}
\newcommand{\muN}{\mu_{\Phi}}
\newcommand{\muA}{\mu_{\widetilde \Phi}}
\newcommand{\muB}{\widetilde \mu_{\Phi}}

\makeatletter
\newcommand{\subjclass}[2][1991]{%
  \let\@oldtitle\@title%
  \gdef\@title{\@oldtitle\footnotetext{#1 \emph{Mathematics subject classification.} #2}}%
}
\newcommand{\keywords}[1]{%
  \let\@@oldtitle\@title%
  \gdef\@title{\@@oldtitle\footnotetext{\emph{Key words and phrases.} #1.}}%
}
\makeatother

\title{On the Local Lipschitz Stability of Bayesian Inverse Problems}
\author{Bj\"orn Sprungk$^\dag$\\
\small
$^\dag$ Institute for Mathematical Stochastics, University of G\"ottingen, Goldschmidtstra\ss e 7, 37077 G\"ottingen, Germany\\
}
\date{\today}

\begin{document}
\maketitle

\begin{abstract}
In this note we consider the stability of posterior measures occuring in Bayesian inference w.r.t.~perturbations of the prior measure and the log-likelihood function.
This extends the well-posedness analysis of Bayesian inverse problems.
In particular, we prove a general local Lipschitz continuous dependence of the posterior on the prior and the log-likelihood w.r.t.~various common distances of probability measures.
These include the total variation, Hellinger, and Wasserstein distance and the Kullback--Leibler divergence.
We only assume the boundedness of the likelihoods and measure their perturbations in an $L^p$-norm w.r.t.~the prior.
The obtained stability yields under mild assumptions the well-posedness of Bayesian inverse problems, in particular, a well-posedness w.r.t.~the Wasserstein distance.
Moreover, our results indicate an increasing sensitivity of Bayesian inference as the posterior becomes more concentrated, e.g., due to more or more accurate data. 
This confirms and extends previous observations made in the sensitivity analysis of Bayesian inference.
\end{abstract}

\\[0.5em]

\noindent
\textbf{Mathematics Subject Classification:} 60B10, 62C10, 62F15, 62G35, 65N21   %65D05, 65D15, 65C30, 60H25.\\[1em]
%==========================================================================================

%==========================================================================================

\section{Introduction}
\label{sec:intro}
In recent years, Bayesian inference has become a popular approach to model and solve inverse problems in various fields of applications, see, e.g., \cite{DashtiStuart2017, KaipioSomersalo2005} for a comprehensive introduction.
Here, noisy observations are used to update the knowledge of unknown parameters from a given \emph{prior} distribution to a resulting \emph{posterior} distribution.
The relation between the parameters and the observable quantities are given by a measurable forwad map which, in combination with an assumed error distribution, determines the likelihood function for the data given the parameter.  
The Bayesian approach is quite appealing and, in particular, yields a well-posed inverse problem \cite{DashtiStuart2017, Hosseini2017, HosseiniNigam2017, Latz2019, Stuart2010, Sullivan2017}, i.e., its unique solution---the posterior distribution---depends (locally Lipschitz) continuously on the observational data and behaves also stable w.r.t.~(numerical) approximations of the forward map.
However, besides the observed data and the employed likelihood model, the subjective choice of the prior distribution significantly affects the outcome of the Bayesian inference, too.
In order to account for that a \emph{robust Bayesian analysis} has emerged, where a class of suitable priors or likelihoods is considered and the range of all resulting posterior quantities or statistics is computed or estimated, see, e.g., \cite{Berger1994, InsuaRuggeri2000} for an introduction. 
Moreover, the well-known \emph{Bernstein--von Mises} theorem \cite{VanDerVaart1998} establishes a kind of ``asymptotic stability'' at least in finite-dimensional spaces.
This theorem tells us that, under suitable assumptions, the posterior measure concentrates around the true parameter, which generates the observations, as more and more data is observed.
This convergence to the truth is called \emph{consistency} and it is independent of the chosen prior measure as long as the true parameter belongs to its support.
However, for posterior measures on infinite-dimensional spaces the situation is far more delicate, and positive as well as negative results for consistency exist, see, e.g., \cite{CastilloNickl2013, CastilloNickl2014, Freedman1999, Leahu2011}.
Furthermore, in \cite{OwhadiScovel2016,OwhadiEtAl2015a,OwhadiEtAl2015b} the authors show an extreme instability of Bayesian inference---called \emph{Bayesian brittleness}---w.r.t.~small perturbations of the likelihood model as well as w.r.t.~classes of priors based on only finitely many pieces of information.
In particular, the range of attainable posterior quantities (e.g., expectations or probabilities) over a class of allowed priors and likelihoods covers the essential (prior) range of the quantity of interest.
This brittleness occurs for arbitrarily many data and arbitrarily small perturbations of the likelihood model.
However, the distance used to measure the size of the perturbations plays a crucial role here as we will discuss later on. 

In this paper we take a slightly different approach than the classical robust Bayesian analysis:
Instead of bounding the resulting posterior range of certain quantities or statistics of interest for a given class of admissible priors or likelihood models, we rather study whether the distance between the posterior measures themselves can be bounded uniformly by a constant multiplied with the distance of the corresponding prior measures or log-likelihood functions.
Thus, the goal is to establish a (local) Lipschitz continuity of the posterior w.r.t.~the prior or the log-likelihood with explicit bounds on the local Lipschitz constant. 
To this end, we employ the following common distances and divergences for (prior and posterior) probability measures: the total variation, Hellinger, and Wasserstein distance as well as the Kullback--Leibler divergence.
Perturbations of the log-likelihood function are measured by suitable $L^p$-norms w.r.t.~the prior measure.
Indeed, under rather mild assumptions we can state the local Lipschitz continuity of posteriors on general Polish spaces w.r.t.~the prior and the log-likehood for all distances and divergences listed above.
On the other hand, our estimates show that the sensitivity of the posterior to perturbations of prior or log-likelihood increases as the posterior concentrates---an observation also made in \cite{DiaconisFreedman1986, GustafsonWasserman1995}.
We discuss this issue and its relation to the Bernstein--von Mises theorem in Section \ref{sec:hellinger} and \ref{sec:wasserstein} in detail. 

As mentioned at the beginning, a local Lipschitz dependence of the posterior measure w.r.t.~the observational data and approximations of the forward map has been proven in \cite{DashtiStuart2017,Stuart2010} for Gaussian and Besov priors and the Hellinger distance.
These results have been generalized to heavy-tailed prior measures in \cite{Hosseini2017, HosseiniNigam2017, Sullivan2017} and a continuous dependence in Hellinger distance was recently shown under substantially relaxed conditions in \cite{Latz2019}.
In the latter work a continuous dependence on the data was also established concerning the Kullback--Leibler divergence and Wasserstein distance between the resulting posteriors.
Moreover, in \cite{MarzoukXiu2009} it is shown that converging approximations of the forward map yield the convergence of the perturbed posteriors to the true posterior in terms of their Kullback--Leibler divergence.
These previous results relate, of course, to our stability statements for perturbed log-likelihoods.

However, the focus of this note is rather on the general structure of the local Lipschitz dependence on the log-likelihood and the prior measure.
In fact, stability w.r.t.~the data or approximations of the forward map follows---under suitable assumptions---from our general results.
Besides that, the local Lipschitz dependence of the posterior on the prior has not been established in the literature to the best of our knowledge.
Furthermore, a stability or well-posedness analysis of Bayesian inverse problems in Wasserstein distance (i.e., perturbations of posterior and prior are measured in Wasserstein distance) has also been missing---the recent results on continuity w.r.t.~the observed data \cite{Latz2019} have been established in parallel to our work. 
This distance is of particular interest for studying the stability w.r.t.~perturbations of the prior measure on infinite-dimensional spaces such as function spaces.
The reason behind is that the total variation and Hellinger distance as well as the Kullback--Leibler divergence obtain their maximum value for mutually singular measures and probability measures on infinite-dimensional spaces tend to be singular---cf., the necessary conditions for Gaussian measures on Hilbert spaces to be equivalent \cite{Bogachev1998, Kuo1975}.
Thus, these distances and divergences may be of little practical use for prior stability whereas the Wasserstein distance of perturbed priors does not rely on their equivalence.
Besides that, the Wasserstein distance has been proven quite flexible and useful for various topics in probability theory such as convergence of Markov processes \cite{HaMaScheu2011} and perturbation theory for Markov chains \cite{RudolfSchweizer2018, MedinaAguayoEtAl2019}.
We establish a first stability analysis of Bayesian inverse problems w.r.t.~the Wasserstein distance and show how the general stability result yields a well-posedness of Bayesian inverse problems in Wasserstein distance.
For the latter we use the same basic assumptions stated in \cite{DashtiStuart2017, Latz2019, Stuart2010} for the well-posedness in Hellinger distance.

In summary, this paper contributes to the stability analysis of Bayesian inference and provides positive statements in a quite general setting.
Our results, although quantitative, are rather of qualitative nature establishing a local Lipschitz stability and, on the other hand, illustrating the increasing sensitivity of the posterior to perturbations in prior or log-likelihood for an increasingly informative likelihood.  
Our setting considers bounded likelihoods which excludes, e.g., the case of infinite-dimensional observations (cf. \cite[Section 3.3]{DashtiStuart2017}).

The outline of the paper is as follows: In the next section we introduce the general setting and the common form of our main results. 
We also discuss their relation to classical robust Bayesian analysis and Bayesian brittleness.
In Section \ref{sec:hellinger} to \ref{sec:wasserstein} we provide the exact statements and proofs for stability in the Hellinger and total variation distance, Kullback--Leibler divergence, and Wasserstein distance.
In particular, we establish in Section \ref{sec:wasserstein} the well-posedness of Bayesian inverse problems in Wasserstein distance.
Furthermore, the relation of the obtained stability statements to the existing literature and results on robust Bayesian analysis is discussed in Section \ref{sec:literature}.
The Appendix includes some more detailed explanations and calculations on the relation between Bayesian brittleness and stability as well as an explicit computation for the Hellinger distance of Gaussian measures on separable Hilbert spaces.

\section{Setting and Main Results}
\label{sec:pre}
Throughout this paper let $(E,d_E)$ be a complete and separable metric space with Borel $\sigma$-algebra $\mc E$ and let $\mc P(E)$ denote the set of all probability measures $\mu$ on $(E,\mc E)$.
We also consider the special case of a separable Hilbert space $\mc H$ with norm $\|\cdot\|_{\mc H}$.

In this paper we focus on \emph{posterior} probability measures $\muN \in \mc P(E)$ of the form
\begin{equation}\label{equ:mu}
	\muN(\d x) 
	\coloneqq \frac1Z \, \exp(- \Phi(x)) \; \mu(\d x),
\end{equation} 
resulting from a \emph{prior} measure $\mu\in \mc P(E)$ and a measurable negative \emph{log-likelihood} $\Phi \colon E \to \bbR_+$, $\bbR_+ \coloneqq [0,\infty)$.
The constant $Z \coloneqq \int_{E} \e^{- \Phi(x)} \; \mu(\d x) \in (0,\infty)$ denotes the normalization constant, sometimes called \emph{evidence}.
The assumption that $\Phi(x) \geq 0$ is convenient and not very restrictive, since any $\muN$ of the form \eqref{equ:mu} with $\Phi\colon E \to \bbR$ being bounded from below, i.e., $\inf_{x\in E} \Phi(x) > -\infty$, can be rewritten as $\muN(\d x) = \exp(- (\Phi(x) - \inf \Phi))/(Z\exp(-\inf \Phi)) \ \mu(\d x)$.

Posterior measures as in \eqref{equ:mu} occur, for example, when we consider the Bayesian approach to inverse problems such as reconstructing an unknown $x^\dagger\in E$ based on noisy data
\[
	y = G(x^\dagger) + \epsilon,
\]
of a measurable forward map $G\colon E \to \bbR^n$ with additive noise $\epsilon \in \bbR^n$.
Inverse problems are typically ill-posed and require some kind of regularization.
In the Bayesian approach we employ a probabilistic regularization, i.e., we incorporate a-priori knowledge about $x^\dagger$ by a prior probability measure $\mu \in \mc P(E)$, and assume a random noise $\varepsilon \sim \nu_\varepsilon \in \mc P(\bbR^n)$, i.e., $\epsilon$ in the equation above is viewed as a realization of the random variable $\varepsilon$.
The Bayesian approach then consists of conditioning the prior measure $\mu$ on observing the data $y \in \bbR^n$ as a realization of the random variable
\[
	Y \coloneqq G(X) + \varepsilon,
	\qquad
	X \sim \mu, \; \varepsilon \sim \nu_{\varepsilon}\quad \mathrm{ stochastically\; independent}.
\]
This results in a conditional or posterior probability measure on $E$.
If the noise distribution allows for a bounded Lebesgue density $\nu_\varepsilon(\d \epsilon) \propto \exp(- \ell(\epsilon)) \d \epsilon$, with $\ell\colon \bbR^n \to \bbR$ being bounded from below, $\inf_{\epsilon \in \bbR^n} \ell(\epsilon) > -\infty$, 
then the posterior measure of $X\sim\mu$ given that $Y = y$ is of the form \eqref{equ:mu} with $\Phi(x) = \Phi(x,y) \coloneqq \ell(y - G(x))$, see, e.g., \cite{DashtiStuart2017, Latz2019, Stuart2010}.
A common noise model is a mean-zero Gaussian noise, i.e., $\nu_\varepsilon = N(0,\Sigma)$, $\Sigma\in\bbR^{n\times n}$ symmetric and positive definite, which yields $\Phi(x) = \Phi(x,y) = \frac 12 |\Sigma^{-1/2}(y -G(x))|^2$ where $|\cdot|$ denotes the Euclidean norm on $\bbR^n$.
For clarity, we omit the dependence of the observational data $y$ in $\Phi$ most of the time.
The main focus of this paper is to study the stability of the posterior $\muN$ w.r.t.~perturbations of the log-likelihood model $\Phi$ as well as the chosen prior measure $\mu$.

An interesting and important special case in practice are Bayesian inverse problems in function spaces, i.e., where $E$ is a separable Hilbert space such as $\mc H = L^2(D)$ with $D\subset \bbR^n$ denoting a spatial domain, cf. Example \ref{exam:one}.
In such situations Gaussian measures $\mu = N(m,C)$ on $\mc H$ are a convenient class of prior measures, see, e.g., \cite{Stuart2010, DashtiStuart2017,LasanenEtAl2014,DunlopEtAl2017}.
Often the mean $m \in \mc H$ and the trace-class covariance operator $C \colon \mc H \to \mc H$ are chosen themselves from parametric classes.
For instance, we may suppose a linear model for the mean $m = \sum_{j=1}^J \theta_j \phi_j$, $\phi_j \in \mc H$, with parameter $\theta = (\theta_1,\ldots, \theta_J)  \in \bbR^J$.
Or, furthermore, we may consider Gaussian prior measures on suitable function spaces $\mc H \subseteq L^2(D)$ with covariance operators belonging to the Mat\'ern class, i.e., $C = C_{\alpha, \beta,\gamma} = \beta(I + \gamma^2 \Delta)^{-\alpha}$ with parameters $\alpha, \beta, \gamma > 0$ and $\Delta$ denoting the Laplace operator, see, e.g., \cite{LasanenEtAl2014,DunlopEtAl2017}.
Since in practice the so-called hyperparameters $\theta$, $\alpha$, $\beta$, and $\gamma$ are often estimated by statistical procedures, a stability of the resulting posterior measure w.r.t.~slightly different means, covariances, or hyperparameters of the corresponding Gaussian prior measures seems highly desirable.
Therefore, we remark in each of the following sections on particular bounds for posteriors resulting from perturbed Gaussian priors.

\begin{exam}\label{exam:one}
A model problem in Bayesian inverse problems is to infer a log conductivity coefficient $u \in C(D)$, $D\subset \bbR^k$, of, e.g., a subsurface layer by noisy measurements of the corresponding potential $p \in H^1(D)$ which is typically described by an elliptic partial differential equation on $D$ of the form
\[
	- \nabla \cdot (\exp(u) \nabla p) = f,
	\qquad
	p\big|_{\partial D} = g,
\]
with suitable source term $f \in L^2(D)$ and boundary data $g\in H^{1/2}(\partial D)$.
In practice there are often (only) finitely many local observations of $p$ as well as of $u$ available, e.g., by measurements at borehole locations.
The observations of $u$ can then be used to derive a Gaussian prior measure $\mu=N(m,C)$ on $\mc H\coloneqq L^2(D)\supset C(D)$ for the (regularizing) Bayesian approach to the underdetermined inverse problem.
However, as explained above the construction of the prior $\mu$ usually involves statistical estimation of hyperparameters for the mean $m$ and the covriance $C$ employing the noisy observations of $u$.
For instance, for Mat\'ern covariances the hyperparameters can be estimated by maximizing the corresponding Gaussian likelihood of the observed values of $u$.
Therefore, also the solution of the Bayesian inverse problem, the posterior measure $\mu_\Phi$ for $u$ resulting from conditioning on the noisy data of $p$, will depend on the accuracy of this statistical estimation.
Our results establish a local Lipschitz dependence of the resulting posterior on the chosen prior, thus, controlling the effect of (inaccurately estimated) priors on posterior decisions. 
\end{exam}

\begin{rem}[Unbounded likelihoods]
Our setting excludes Bayesian inverse problems with unbounded likelihood functions $\e^{-\Phi}\colon E\to(0,\infty)$.
Such likelihoods can appear in the additive noise model described above if the probability density function of the noise distribution $\nu_\varepsilon$ is not bounded---for instance, if one component of $\varepsilon$ follows a chi-squared distribution with degree of freedom one or a beta distribution with both shape parameters equal to one half.
Moreover, also the setting of infinite-dimensional data is excluded: Consider an observable $Y \coloneqq G(X) + \varepsilon$ where $G\colon E \to \mc H$ is a measurable mapping into an infinite-dimensional separable Hilbert space $\mc H$ and $\varepsilon \sim N(0, Q)$ is a mean-zero Gaussian random variable in $\mc H$ with covariance operator $Q\colon \mc H\to\mc H$. 
If the range of $G$ is a subset of the range of $Q$, i.e., $\rg G \subset \rg Q$, then the resulting posterior for $X\sim \mu$ given a realization $Y = y$, $y\in\mc H$, can be shown to be of the form \eqref{equ:mu}, cf. \cite[Section 3.3]{DashtiStuart2017}.
In this case $\Phi$ is given by 
\[
	\Phi(x) \coloneqq \frac 12 \|Q^{-1/2}G(x)\|_{\mc H}^2 - \langle Q^{-1}G(x), y\rangle_{\mc H},
	\qquad
	x \in E,
\]
where $\langle \cdot, \cdot\rangle_{\mc H}$ denotes the inner product in $\mc H$ and where we assume that $\exp(-\Phi) \in L^1_\mu(\bbR)$ for the moment. 
However, if $y \notin \rg Q^{1/2}$, then $\Phi(x)$ can, in general, not be bounded from below---since the lower bound is, formally, $- \frac 12 \|Q^{-1/2}y\|_{\mc H}$.

On the other hand, more general noise models than additive noise are covered. 
For instance, if $G\colon E\to\bbR$ is measurable and we can observe $Y \coloneqq \varepsilon G(X)$ with $\varepsilon \sim N(0,\sigma^2)$ being stochastically independent of $X\sim \mu$, then the resulting posterior given $Y=y$ is again of the form \eqref{equ:mu} with $\Phi(x) = \frac 12 y^2/(\sigma^2 G^2(x)) \geq 0$.
\end{rem}

%\subsection{Main Results}
\paragraph{Main Results.}
We are interested in the stability of the posterior measure $\muN\in\mc P(E)$ w.r.t.~perturbations of the negative log-likelihood function $\Phi\colon E \to [0,\infty)$ and the prior measure $\mu\in\mc P(E)$. 
The former includes for $\Phi(x) = \ell(y - G(x))$ perturbations of the observed data $y \in \bbR^n$ or the forward map $G$, e.g., due to numerical approximations of $G$, cf.~Remark \ref{rem:data_G} below.
We then bound the difference between the original posterior $\muN$ and two kinds of perturbed posteriors:
\begin{enumerate}
\item
perturbed posteriors $\muA \in \mc P(E)$ resulting from perturbed log-likelihood functions $\widetilde \Phi \colon E \to \bbR$, i.e., 
\begin{equation}\label{equ:mu_tilde_1}
	\muA(\d x) 
	\coloneqq \frac1{\widetilde Z} \, \e^{- \widetilde \Phi(x)} \; \mu(\d x),
	\qquad
	\widetilde Z \coloneqq \int_{E} \e^{- \widetilde \Phi(x)} \; \mu(\d x),
\end{equation}

\item
perturbed posteriors $\muB \in \mc P(E)$ resulting from perturbed prior measures $\widetilde \mu \in \mc P(E)$, i.e.,
\begin{equation}\label{equ:mu_tilde_2}
	\muB(\d x) 
	\coloneqq \frac1{\widetilde Z} \, \e^{- \Phi(x)} \; \widetilde \mu(\d x),
	\qquad
	\widetilde Z \coloneqq \int_{E} \e^{- \Phi(x)} \; \widetilde \mu(\d x).
\end{equation}
\end{enumerate}
In particular, we prove the local Lipschitz continuity of $\muN \in \mc P(E)$ w.r.t.~the log-likelihood $\Phi \in L^p_{\mu}(\bbR_+)$ and the prior $\mu \in \mc P(E)$ in several common distances and divergences for probability measures $d\colon \mc P(E)\times \mc P(E) \to [0,\infty]$.
Here, $L^p_{\mu}(\bbR_+)$ denotes the set of non-negative functions which are $p$-integrable w.r.t.~the prior measure $\mu$.
For $d$ we consider the total variation distance, the Hellinger distance, the Kullback--Leibler divergence, and the Wasserstein distance.
Given suitable assumptions our results are then of the following form:

\begin{enumerate}
\item
For a given prior $\mu \in \mc P(E)$ and log-likelihoods $\Phi, \widetilde \Phi \in L^p_{\mu}(\bbR_+)$ with suitable $p\in\bbN$, there exists a constant $C_{\mu,\Phi} < \infty$ and a $q \in \bbN$ such that
\begin{equation*}
	d(\muN, \muA)
	\leq
	\frac{C_{\mu,\Phi}}{\min(Z,\widetilde Z)^q} \,
	\|\Phi - \widetilde \Phi\|_{L^p_{\mu}}.
\end{equation*}
This yields a local Lipschitz continuity as follows: For any $\Phi \in L^p_{\mu}(\bbR_+)$ and any radius $r>0$ there exists again a constant $C_{\mu,\Phi}(r) < \infty$ such that
\begin{equation}\label{equ:cont_Phi}
	d(\muN, \muA)
	\leq
	C_{\mu, \Phi}(r) \, \|\Phi - \widetilde \Phi\|_{L^p_{\mu}}
	\qquad
	\forall \widetilde \Phi\in L^p_{\mu}(\bbR_+)\colon \|\widetilde \Phi - \Phi\|_{L^p_{\mu}} \leq r.
\end{equation}
The particular values for $C_{\mu, \Phi}(r)$ and $p$ for each of the studied distances and divergences are given in Table \ref{tab:Phi} where for the Wasserstein distance we require $|\mu|_{\mc P^2}^2 \coloneqq \inf_{x_0\in E} \int_E d_E^2(x,x_0) \, \mu(\d x) < \infty$.

\begin{table}[t]
\begin{center}
\begin{tabular}{lcccc}
\toprule[1.25pt]
$d$ & Total Variation & Hellinger & Kullback--Leibler & $1$-Wasserstein\\ \midrule
%\mr
$C_{\mu,\Phi}(r)$ & $Z^{-1}$ & $\exp(\|\Phi\|_{L^1_\mu}+r)$ & $2\exp(\|\Phi\|_{L^1_\mu}+r)$ & $2|\mu|_{\mc P^2}\exp(2\|\Phi\|_{L^1_\mu}+2r)$\\
$p$ & $1$ & $2$ & $1$ & $2$\\
\bottomrule[1.25pt]
\end{tabular}
\end{center}
\caption{\label{tab:Phi} Local Lipschitz constant $C_{\mu,\Phi}(r)$ of the mapping $L^p_{\mu}(\bbR_+)\ni\Phi \mapsto \mu_\Phi \in (\mc P(E),d)$ in \eqref{equ:cont_Phi} w.r.t.~various distances and divergences $d$; here $r>0$ denotes the radius of the local neighborhood of $\Phi$ in $L^p_{\mu}(\bbR_+)$.}
\end{table}

\item
Similarly, for a given measurable $\Phi\colon E\to[0,\infty)$ and suitable priors $\mu,\widetilde \mu \in \mc P(E)$ there exists a constant $C_{\mu,\Phi} < \infty$ and a $q \in \bbN$ such that
\begin{equation*}%\label{equ:cont_Phi}
	d(\muN, \muB)
	\leq
	\frac{C_{\mu,\Phi}}{\min(Z,\widetilde Z)^q} \,
	\frac{d(\mu, \widetilde \mu) + d(\widetilde \mu, \mu)}{2},
\end{equation*}
where $\frac{d(\mu, \widetilde \mu) + d(\widetilde \mu, \mu)}{2} \neq d(\mu, \widetilde \mu)$ in case of the Kullback--Leibler divergence.
Again, this yields a local Lipschitz continuity: For any $\mu \in \mc P(E)$ there exists for each radius $0< r < R_{\mu,\Phi}$ a constant $C_{\mu,\Phi}(r) < \infty$ such that
\begin{equation}\label{equ:cont_mu}
	d(\muN, \muB)
	\leq
	C_{\mu, \Phi}(r) \, \frac{d(\mu, \widetilde \mu) + d(\widetilde \mu, \mu)}{2}
	\qquad
	\forall\, \mathrm{suitable}\, \widetilde \mu\colon d(\mu,\widetilde \mu)\leq r.
\end{equation}
The particular values for $C_{\mu, \Phi}(r)$ and the radius bound $R_{\mu,\Phi}$ as well as corresponding conditions for each of the studied distances and divergences are given in Table \ref{tab:mu}---here, $\Lip(\e^{-\Phi})$ denotes the global Lipschitz constant of $\e^{-\Phi}\colon E\to(0,1]$ w.r.t.~$d_E$.

\begin{table}[t]
\begin{center}
\begin{tabular}{lcccc}
\toprule[1.25pt]
$d$ & Total Variation & Hellinger & Kullback--Leibler & $1$-Wasserstein.\\
\midrule
$C_{\mu,\Phi}(r)$ & $2/Z$ & $2/(Z-2r) $ & $2/(Z-\sqrt{2r}) $ & $(1+D\Lip(\e^{-\Phi}))^2/(Z-\Lip(\e^{-\Phi})r)$\\
$R_{\mu,\Phi}$ & $+\infty$ & $Z/2$ & $Z^2/2$ & $Z/\Lip(\e^{-\Phi})$\\
\midrule
Conditions & --- & --- & $\mu, \widetilde \mu$ equivalent & $d_E$ bounded by $D$, $\e^{-\Phi}$ Lipschitz\\
\bottomrule[1.25pt]
\end{tabular}
\end{center}
\caption{\label{tab:mu} Local Lipschitz constant $C_{\mu,\Phi}(r)$ of the mapping $\mu \mapsto \mu_\Phi$ in \eqref{equ:cont_mu} w.r.t.~various distances and divergences $d$ and necessary conditions for \eqref{equ:cont_mu} to hold; here $R_{\mu,\Phi}$ denotes the upper limit for the radius $r$ of the local neighborhood of $\mu$ w.r.t.~$d$.}
\end{table}

\item
In both cases (i) and (ii) the estimated local Lipschitz constant $C_{\mu, \Phi}(r)$ grows as the normalization constant $Z$ of $\muN$ decays or $\|\Phi\|_{L^1_\mu}$ increases, respectively.
For the local Lipschitz dependence on the prior we even have that the maximal local neighborhood of $\mu$, for which our Lipschitz bound holds, shrinks as $Z\to0$.
Thus, the sensitivity w.r.t.~perturbtations of prior or log-likelihood increases, in general, if the normalization constant $Z$ of $\muN$ decreases due to, for instance, a higher concentration or localization of the posterior measure.
\end{enumerate}
Thus, besides positive stability results our bounds suggest an increasing sensitivity of the posterior w.r.t.~perturbations of the prior or log-likelihood for increasingly informative $\Phi$, e.g., due to more or more precise observations employed in the Bayesian inference.

\begin{rem}\label{rem:PM_bounds}
The considered total variation distance $d_\mathrm{TV}$, Hellinger distance $d_\mathrm{H}$, and Kullback--Leibler divergence $d_\mathrm{KL}$ are related in the following way:
\begin{equation}\label{equ:PM_bounds}
	\frac 12 d^2_\mathrm{H}(\mu,\widetilde \mu) 
	\leq d_\mathrm{TV}(\mu,\widetilde \mu) 
	\leq d_\mathrm{H}(\mu,\widetilde \mu) 
	\leq \sqrt{d_\mathrm{KL}(\mu||\widetilde \mu)},
	\qquad
	\mu,\widetilde \mu \in \mc P(E),
\end{equation}
see \cite{GibbsSu2002}. Moreover, on bounded metric spaces $(E,d_E)$ with $d_E \leq D$, $D\in\bbR$, we have for the $1$-Wasserstein distance $\W_1(\mu,\widetilde \mu) \leq D\, d_\mathrm{TV}(\mu,\widetilde \mu) $, see \cite{GibbsSu2002}. 
However, the established local Lipschitz bounds in Table \ref{tab:Phi} and \ref{tab:mu} are derived by studying each distance individually, since this allowed for sharper estimates.
For example, a Lipschitz continuity of $\Phi \mapsto \mu_\Phi$ in $d_\mathrm{KL}$ would only imply a H\"older continuity in $d_\mathrm{TV}$ or $d_\mathrm{H}$ by \eqref{equ:PM_bounds}.
\end{rem} 

\begin{rem}\label{rem:data_G}
The bound $d(\muN, \muA) \leq C_{\mu, \Phi}(r) \|\Phi - \widetilde \Phi\|_{L^p_{\mu}}$ can usually be used to prove local Lipschitz continuous dependence of the posterior on the data or show stability w.r.t.~numerical approximations, say $G_h$, of the forward map $G\colon E \to \bbR^n$. 
For example, for an additive Gaussian noise $\varepsilon \sim N(0,\Sigma)$ we can set $\widetilde \Phi(x) \coloneqq \frac 12 |y- G_h(x)|^2$ and obtain
\[
	|\Phi(x) - \widetilde \Phi(x)|
	\leq
	\left(2|y|_{\Sigma^{-1}} + |G(x)|_{\Sigma^{-1}} + |G_h(x)|_{\Sigma^{-1}} \right)
	\left| G(x) - G_h(x) \right|.
\]
If $G, G_h \in L^{2p}_{\mu}(\bbR^n)$, then the Cauchy--Schwarz inequality yields
\[
	d(\muN, \muA)
	\leq
	C_{\mu, \Phi}(r)
	C_{\Sigma^{-1}}
	\left(|y| + \|G\|_{L^{2p}_{\mu}} + \|G_h\|_{L^{2p}_{\mu}}\right)
	\|G - G_h\|_{L^{2p}_{\mu}}.
\]
Analogous expressions can be obtained for the case of perturbed data $\widetilde y\in\bbR^n$, i.e.,  for $\widetilde \Phi(x) \coloneqq \frac 12 |\widetilde y- G(x)|^2$.
\end{rem}

%==========================================================================================
%
% STABILITY WRT HELLINGER	
%
\section{Stability in Hellinger and Total Variation Distance}
\label{sec:hellinger}
First, we study the continuity of the posterior measure $\muN\in\mc P(E)$ w.r.t.~the log-likelihood function $\Phi \colon E \to [0,\infty)$ and the prior measure $\mu\in\mc P(E)$ in the Hellinger distance
\[
	d^2_\mathrm{H}(\mu, \widetilde \mu)
	\coloneqq
	\int_{E} \left( \sqrt{\frac{\d\mu}{\d\nu}(x)} - \sqrt{\frac{\d\widetilde \mu}{\d\nu}(x)} \right)^2 \nu(\d x),
	\qquad
	\mu,\widetilde \mu \in \mc P(E).
\]
Here, $\nu$ denotes an arbitrary measure on $E$ dominating $\mu$ and $\widetilde \mu$, e.g., $\nu = \frac12\mu + \frac12\widetilde \mu$.
The Hellinger distance is topologically equivalent to the total variation distance
% , i.e., we have for $\mu,\widetilde\mu\in\mc P(E)$
\[
	d_\mathrm{TV}(\mu, \widetilde \mu)
	\coloneqq
	\sup_{A \in \mc E} |\mu(A) - \widetilde \mu(A)|
	=
	\frac 12
	\int_{E} \left|\frac{\d\mu}{\d\nu}(x) - \frac{\d\widetilde \mu}{\d\nu}(x) \right| \nu(\d x),
\]
see \eqref{equ:PM_bounds},  
%%where $\mu,\widetilde \mu \in \mc P(E)$ are again dominated by $\nu$, 
%due to
%\begin{equation}\label{equ:tv_hell}
%	\frac12 d^2_\mathrm{H}(\mu, \widetilde \mu) \leq d_\mathrm{TV}(\mu, \widetilde \mu) \leq d_\mathrm{H}(\mu, \widetilde \mu),
%	\qquad
%	\mu,\widetilde \mu \in \mc P(E),
%\end{equation}
%see, e.g., \cite{GibbsSu2002}, 
but it yields also continuity of the moments of square-integrable functions, see \cite[Theorem 21]{DashtiStuart2017}.
We investigate now the stability of the posterior $\muN$ w.r.t.~$\Phi$ in Hellinger distance.
% observed data and its stability w.r.t.~approximations of the forward map $F$.
The related issue of stability w.r.t.~the data $y$ and numerical approximations $\Phi_h$ of $\Phi$, cf. Remark \ref{rem:data_G}, was already established for Bayesian inverse problems with additive Gaussian noise and a Gaussian prior $\mu$ by Stuart \cite{Stuart2010} and recently extended by Dashti and Stuart \cite{DashtiStuart2017} and Latz \cite{Latz2019} to a more general setting.
Moreover, in \cite[Section 4]{StuartTeckentrup2018} we already find a similar result under slightly different assumptions.
Nonetheless, we state the theorem and proof for completeness.
%The following theorem emphasizes the underlying mathematics behind the results in \cite{Stuart2010} and \cite{DashtiStuart2017}.

\begin{theo}\label{theo:hell_Phi}
Let $\mu \in \mc P(E)$ and $\Phi, \widetilde \Phi\colon  E \to \bbR$ belong to $L^2_{\mu}(\bbR)$ with $\essinf_{\mu} \Phi = 0$.
Then, we have for the two probability measures $\muN, \muA \in \mc P(E)$ given by \eqref{equ:mu} and \eqref{equ:mu_tilde_1}, respectively,
\begin{equation}\label{bayes:equ:mu_Phi_cont}
	d_\mathrm{H}(\muN, \muA) \leq \frac{\exp(- [\essinf_{\mu} \widetilde \Phi]_- )}{\min(Z,\widetilde Z)}\,\|\Phi - \widetilde \Phi\|_{L^2_{\mu}},
\end{equation}
where $[t]_- \coloneqq \min(0,t)$ for $t\in\bbR$.
\end{theo}
\begin{proof}
Analogously to the proof of \cite[Theorem 4.6]{Stuart2010} or \cite[Theorem 4.2]{StuartTeckentrup2018}  we start with %using $(a-b)^2 \leq 2(a-c)^2+2(b-c)^2$ and 
\begin{eqnarray*}
	d_\mathrm{H}(\muN, \muA)^2
	& = & \int_{E}\left( \frac{\e^{-\Phi(x)/2}}{\sqrt{Z}} -\frac{\e^{-\widetilde \Phi(x)/2}}{\sqrt{\widetilde Z}}\right)^2 \, \mu(\d x)\\
	& \leq 2 & \int_{E}\left( \frac{\e^{-\Phi(x)/2}}{\sqrt{Z}} - \frac{\e^{-\widetilde \Phi(x)/2}}{\sqrt{Z}}\right)^2 + \left( \frac{\e^{-\widetilde \Phi(x)/2}}{\sqrt{Z}} - \frac{\e^{-\widetilde \Phi(x)/2}}{\sqrt{\widetilde Z}}\right)^2 \, \mu(\d x)
	= I_1 + I_2,
\end{eqnarray*}
where
\[
	I_1
	\coloneqq \frac 2Z \int_{E} \left( \e^{-\Phi(x)/2} - \e^{-\widetilde \Phi(x)/2} \right)^2 \, \mu(\d x),
	\quad
	I_2
	\coloneqq 
	2 \widetilde Z \left(\frac 1{\sqrt{Z}}-\frac 1{\sqrt{\widetilde Z}}\right)^2
	=
	\frac2Z \left( \sqrt{\widetilde Z} - \sqrt{Z}\right)^2.
\]
Since $|\e^{-t} - \e^{-s}| = \e^{-\min(t,s)} |1 - \e^{-|t-s|}| \leq \e^{-\min(t,s)} |t-s|$ for any $t,s\in\bbR$, we obtain
\[
	I_1
	\leq \frac {2\exp(- [\essinf_{\mu} \widetilde \Phi]_- )}{Z} \int_{E} \frac{\left| \Phi(x) - \widetilde \Phi(x) \right|^2}{4} \, \mu(\d x) 
	= \frac {\exp(- [\essinf_{\mu} \widetilde \Phi]_- )}{2Z} \|\Phi-\widetilde \Phi\|_{L^2_{\mu}}^2
\]
and since $|t^{1/2} - s^{1/2}|\leq \frac 12 \min(t,s)^{-1/2} \, |t-s|$ for $t,s>0$ we have
\[
	I_2
	\leq \frac{1}{2Z\, \min(Z, \widetilde Z)} \, |Z - \widetilde Z|^2.
\]
Now, as for $I_1$ we obtain
\[
	|Z - \widetilde Z|
	\leq
	\int_{E} \left| \e^{-\Phi(x)} - \e^{-\widetilde \Phi(x)} \right| \, \mu(\d x)
	\leq \exp(- [\essinf_{\mu}\widetilde \Phi]_- ) \|\Phi-\widetilde \Phi\|_{L^1_{\mu}}
\]
and due to $Z\leq 1$ we have
\[
	\frac 1{2Z} + \frac{1}{2Z\min(Z, \widetilde Z)} 
	\leq \frac 1{2\min(Z, \widetilde Z)^2} + \frac{1}{2\min(Z, \widetilde Z)^2} 
	= \frac1{\min(Z, \widetilde Z)^2}
\]
which concludes the proof.
\end{proof}

For the term $\min(Z,\widetilde Z)$ appearing in the estimate \eqref{bayes:equ:mu_Phi_cont} we provide the following lower bound.
\begin{propo}\label{propo:Z_bound}
Let $\mu \in \mc P(E)$ and $\Phi, \widetilde \Phi\colon  E \to \bbR$ belong to $L^1_{\mu}(\bbR)$.
Considering the probability measures $\muN, \muA \in \mc P(E)$ given by \eqref{equ:mu} and \eqref{equ:mu_tilde_1} we have %for their normalizing constants $Z,\widetilde Z$
\begin{equation}\label{equ:Z_bound}
	\min(Z,\widetilde Z)
	\geq
%	\exp\left(- \max\left(\|\Phi\|_{L^1_\mu}, \|\widetilde \Phi\|_{L^1_\mu}\right)\right)
%	\geq
	\exp\left(- \|\Phi\|_{L^1_\mu} - \|\Phi - \widetilde \Phi\|_{L^1_\mu}\right).
\end{equation}
\end{propo}
\begin{proof}
By Jensen's inequality and the convexity of $t \mapsto \exp(-t)$ we have
\[
	Z = \int_E \exp(-\Phi(x)) \ \mu(\d x)
	\geq
	\exp\left( - \int_E \Phi(x) \ \mu(\d x)\right)
	\geq
	\exp\left(- \|\Phi\|_{L^1_\mu}\right).	
\]
Analogously, $\widetilde Z \geq \exp\left(- \|\widetilde\Phi\|_{L^1_\mu}\right)$.	
The statment follows by the triangle inequality.
\end{proof}
Combining Proposition \ref{propo:Z_bound} and Theorem \ref{theo:hell_Phi} we obtain for a given prior $\mu \in \mc P(E)$ the \emph{local Lipschitz continuity} of the mapping $L^2_\mu(\bbR_+) \ni \Phi \mapsto \mu_\Phi \in \mc P(E)$ w.r.t.~the Hellinger distance.
In particular, for a given $\Phi \in L^2_\mu(\bbR_+)$ and any $r>0$ there exists a constant $C_{\Phi,\mu}(r) \coloneqq \exp(\|\Phi\|_{L^1_\mu} + r)<\infty$ such that
\[
	d_\mathrm{H}(\muN, \muA)
	\leq C_{\Phi,\mu}(r)\, \|\widetilde \Phi - \Phi\|_{L^2_{\mu}}
	\qquad
	\forall \widetilde \Phi\in L^2_{\mu}(\bbR_+)\colon \|\widetilde \Phi - \Phi\|_{L^2_{\mu}} \leq r.
\]

All results in the remainder of the paper will be of the form as in Theorem \ref{theo:hell_Phi}: We bound the distance of the posteriors by a constant times the distance of the log-likelihoods or priors where the constant depends on $\min(Z,\widetilde Z)$ which can be bounded uniformly for all sufficiently small perturbations.

For stability w.r.t.~different priors we get the following result.

\begin{theo}\label{theo:hell_prior}
Let $\mu,\widetilde \mu \in \mc P(E)$ and $\Phi\colon E \to [0,\infty)$ be measurable.
Then, for $\muN, \muB \in \mc P(E)$ as in \eqref{equ:mu} and \eqref{equ:mu_tilde_2}, respectively, we have
\[
	d_\mathrm{H}(\muN, \muB) \leq \frac{2}{\min(Z, \widetilde Z)}  \, d_\mathrm{H}(\mu, \widetilde \mu),
	\qquad
	|Z - \widetilde Z| \leq 2 d_\mathrm{H}(\mu, \widetilde \mu).
\]
\end{theo}
\begin{proof}
Let $\rho(x) \coloneqq \frac{\d \mu}{\d \nu}(x)$ and $\widetilde \rho(x) \coloneqq \frac{\d \widetilde \mu}{\d \nu}(x)$ denote the densities of $\mu$ and $\widetilde \mu$ w.r.t.~a dominating $\nu \in \mc P(E)$.
Then, we have
\[
	\frac{\d \muN}{\d \nu}(x) 
	= \frac{\d \muN}{\d \mu}(x) \, \frac{\d \mu}{\d \nu}(x) 
	= 
	\frac{\e^{-\Phi(x)}}{Z} \rho(x),
	\qquad
	\frac{\d \muB}{\d \nu}(x) 
	= \frac{\d \muB}{\d \widetilde \mu}(x) \, 	\frac{\d \widetilde \mu}{\d \nu}(x) 
	= \frac{\e^{-\Phi(x)}}{\widetilde Z} \widetilde \rho(x),
\]
where $Z = \int_{E} \e^{-\Phi(x)}\, \rho(x)\, \nu(\d x)$ and $\widetilde Z = \int_{E} \e^{-\Phi(x)}\, \widetilde \rho(x)\, \nu(\d x)$.
We obtain analogously to Theorem \ref{theo:hell_Phi}
\begin{eqnarray*}
	d^2_\mathrm{H}(\muN, \muB)
	& = &\int_{E} \left(\e^{-\Phi(x)/2}\sqrt{\frac{\rho(x)}{Z}} -\e^{-\Phi(x)/2}\sqrt{\frac{\widetilde \rho(x)}{\widetilde Z}}\right)^2 \, \nu(\d x)\\
	& \leq & 2 \int_{E} \e^{-\Phi(x)} \left(\frac{\sqrt{\rho(x)}}{\sqrt{Z}} - \frac{\sqrt{\widetilde \rho(x)}}{\sqrt{Z}} \right)^2 + \e^{-\Phi(x)} \left(\frac{\sqrt{\widetilde \rho(x)}}{\sqrt{Z}} - \frac{\sqrt{\widetilde \rho(x)}}{\sqrt{\widetilde Z}}\right)^2 \, \nu(\d x)\\
	& = & I_1 + I_2
\end{eqnarray*}
where
\[
	I_1
	\coloneqq \frac{2}{Z} \int_{E} \e^{-\Phi(x)} \left( \sqrt{\rho(x)} - \sqrt{\widetilde \rho(x)}\right)^2 \, \nu(\d x),
	\qquad
	I_2
	\coloneqq 
	\frac{2}{Z} \left( \sqrt{\widetilde Z} - \sqrt{Z}\right)^2.
\]
Due to $\Phi(x) \geq 0$ we get
\[
	I_1 
	\leq
	\frac 2{Z} \int_{E} \left( \sqrt{\rho(x)} - \sqrt{\widetilde \rho(x)}\right)^2 \, \nu(\d x)
	=\frac 2{Z} d^2_\mathrm{H}(\mu, \widetilde\mu).
\]
Moreover, as in the proof of Theorem \ref{theo:hell_Phi}, we have $I_2 \leq \frac 1{2Z \min(Z,\widetilde Z)} |Z-\widetilde Z|^2$ and due to \eqref{equ:PM_bounds}
\[
	|Z-\widetilde Z|
	\leq \int_{E} \e^{-\Phi(x)} |\rho(x) - \widetilde \rho(x)|\ \nu(\d x)
	\leq 2 d_\mathrm{TV}(\mu, \widetilde \mu) 
	\leq 2 d_\mathrm{H}(\mu, \widetilde \mu). 	
\]
Hence, since $Z \leq 1$ we obtain
\[
	d^2_\mathrm{H}(\muN, \muB)
	\leq 
	I_1+I_2
	\leq
	 \left(\frac 2{Z} + \frac 2{Z \min(Z, \widetilde Z)}\right)
	d^2_\mathrm{H}(\mu, \widetilde \mu)
	\leq
	 \frac{4 d^2_\mathrm{H}(\mu, \widetilde \mu)}{\min(Z,\widetilde Z)^2}.
\]
\end{proof}

Concerning the \emph{local Lipschitz continuity} of the mapping $\mc P(E) \ni \mu \mapsto \mu_\Phi \in \mc P(E)$, $\Phi\colon E\to[0,\infty)$, w.r.t.~the Hellinger distance, we obtain the following due to Theorem \ref{theo:hell_prior}: Given a $\mu \in \mc P(E)$ and a radius $0\leq r < \frac 12Z = \frac 12\int_E\e^{-\Phi}\, \d\mu$ there exists a constant $C_{\Phi,\mu}(r) \coloneqq \frac{2}{Z - 2r}$ such that 
\[
	d_\mathrm{H}(\muN, \muB)
	\leq
	C_{\Phi,\mu}(r)\, d_\mathrm{H}(\mu, \widetilde \mu)
	\qquad
	\forall \widetilde\mu\in\mc P(E)\colon d_\mathrm{H}(\mu,\widetilde\mu) \leq r,
\]
due to $\min(Z,\widetilde Z) \geq Z - |Z-\widetilde Z| \geq Z - 2d_\mathrm{H}(\mu,\widetilde\mu)$. 
Hence, in comparison to the local Lipschitz continuity of $\Phi \mapsto \mu_\Phi$ discussed above the radius of the local neighborhood of $\mu$ also depends on $\mu$.

For stability in total variation distance we could simply use the relation \eqref{equ:PM_bounds} between $d_\mathrm{TV}$ and $d_\mathrm{H}$ combined with Theorem \ref{theo:hell_Phi} and Theorem \ref{theo:hell_prior}.
However, sharper bounds are obtained by adopting the proofs of Theorem \ref{theo:hell_Phi} and Theorem \ref{theo:hell_prior},  accordingly.
\begin{theo}\label{theo:TV}
Let $\mu\in \mc P(E)$ and $\Phi, \widetilde \Phi \in L^1_\mu(\bbR)$ with $\essinf_{\mu} \Phi = 0$.
Then, we have for $\muN, \muA \in \mc P(E)$ as in \eqref{equ:mu} and \eqref{equ:mu_tilde_1}, respectively,
\[
	d_\mathrm{TV}(\muN, \muA) \leq \frac{\exp(- [\essinf_{\mu} \widetilde \Phi]_- )}{Z}\,\|\Phi - \widetilde \Phi\|_{L^1_{\mu}}.
\]
Moreover, for $\mu, \widetilde \mu\in\mc P(E)$ and measurable $\Phi\colon E\to[0,\infty)$ we have for $\muN, \muB \in \mc P(E)$ as in \eqref{equ:mu} and \eqref{equ:mu_tilde_2}, respectively,
\[
	d_\mathrm{TV}(\muN, \muB) \leq \frac{2}{Z}\,d_\mathrm{TV}(\mu, \widetilde \mu).
\]
\end{theo}

\begin{rem}[Increasing sensitivity]\label{rem:BvM}
The bounds established in Theorem \ref{theo:hell_Phi}, \ref{theo:hell_prior}, and \ref{theo:TV} involve the inverse of the normalization constant $Z$ of $\muN$.
This suggests that Bayesian inference becomes increasingly sensitive to pertubations of the log-likelihood or prior as the posterior $\muN$ concentrates due to more or more accurate data.
This may seem counterintuitive given the well-known \emph{Bernstein--von Mises theorem} \cite{VanDerVaart1998, KleijnVanDerVaart2012} in asymptotic Bayesian statistics: Under suitable conditions the posterior measure concentrates around the true, data-generating $x^\dagger \in E$ in the large data limit.
This statement holds independently of the particular prior $\mu$ as long as $x^\dagger$ belongs to the support of the measure $\mu\in\mc P(E)$, i.e., as long as $x^\dagger \in \supp \mu $. 
However, the latter resolves the alleged contradiction: Given a suitable infinite space $E$ and a non-atomic prior $\mu \in \mc P(E)$---i.e., for each $x \in E$ we have for balls $B_r(x)\coloneqq\{y\in E\colon d_E(x,y)\leq r\}$ that $\lim_{r\to0} \mu(B_r(x)) = 0$\footnote{Similar requirements were essential for the brittleness results in \cite{OwhadiEtAl2015a, OwhadiEtAl2015b}.}---we can construct for any $\epsilon > 0$ a perturbed prior $\widetilde \mu$ with $d_\mathrm{TV}(\mu,\widetilde \mu) \leq \epsilon$ but $\widetilde \mu(B_r(x^\dagger)) = 0$ for a sufficiently small radius $r = r(\epsilon)>0$.
Thus, $\muN$ concentrates around $x^\dagger$ and $\muB$ around another $x^\star \in \supp \widetilde \mu$, see \cite{KleijnVanDerVaart2012}, and their total variation distance will tend to $1$ since $d_E(x^\dagger, x^\star)\geq r >0$. 
Similar arguments also apply to perturbations of the likelihood function, since $x^\dagger$ is typically the minimizer of the log-likelihood $\Phi$ on $\supp \mu$, and, therefore, we can construct perturbed $\widetilde \Phi$ with a different minimizer $x^\star \neq x^\dagger$ but with arbitrarily small $L^1_\mu$-distance $\|\Phi - \widetilde \Phi\|_{L^1_{\mu}}$.
Thus, it is indeed the case that Bayesian inference becomes more sensitive to perturbations of the log-likelihood or the prior as the amount of data or its accuracy increases.
This also holds for other divergences and distances, cf. Remark \ref{rem:Sens_Wasserstein}.
\end{rem}

\begin{rem}[Stability w.r.t.~Gaussian priors]\label{rem:Hell_Gauss}
Concerning Gaussian priors $\mu = N(m,C)$ and $\widetilde \mu = N(\widetilde m, \widetilde C)$ on a separable Hilbert space $\mc H$ with norm $\|\cdot\|_{\mc H}$ we can bound the Hellinger distance of the resulting posteriors by Theorem \ref{theo:hell_prior}.
In order to obtain a non-trivial bound, we require that $\mu$ and $\widetilde \mu$ are absolutely continuous w.r.t.~each other, i.e., that $m-\widetilde m \in \rg C^{1/2} = \rg \widetilde C^{1/2}$ and $C^{-1/2}\widetilde C C^{-1/2} - I$ is a Hilbert--Schmidt operator on $\mc H$, see, e.g., \cite[Corollary 6.4.11]{Bogachev1998} or \cite[Section II.3]{Kuo1975}.
Assuming furthermore that $T\coloneqq C^{-1/2}\widetilde C C^{-1/2}$ is a positive definite operator on $\mc H$, we can then use the exact expressions for the Hellinger distance of equivalent Gaussian measures:
\begin{eqnarray*}
	d^2_\mathrm{H}(N(m, C), N(\widetilde m, C))
	& = &
	2
	-
	2\exp\left(-\frac 18 \|C^{-1/2} (m-\widetilde m) \|_{\mc H}^2\right),\\
	d^2_\mathrm{H}(N(m, C), N(m, \widetilde C))
	& = &
	2
	-
	2 \left[ \det\left(\frac 12 \sqrt T + \frac 12 \sqrt{T^{-1}}\right)\right]^{-1/2}.
\end{eqnarray*}
We provide a detailed derivation of these formulas in Appendix \ref{app:Hell} and only make the following remarks here: (a) the inverse $T^{-1}$ exists and is bounded on $\mc H$, since $T$ is positive definite and $T-I$ is Hilbert--Schmidt, i.e., the smallest eigenvalue of $T$ is bounded away from zero; (b) the determinant $\det\left(\frac 12 \sqrt T + \frac 12 \sqrt{T^{-1}}\right)$ is well-defined as a Fredholm determinant, since $I - (\frac 12 \sqrt T + \frac 12 \sqrt{T^{-1}})$ is trace-class, see Appendix \ref{app:Hell}; and (c) we have $\det\left(\frac 12 \sqrt T + \frac 12 \sqrt{T^{-1}}\right)\geq1$ due to $\frac {\sqrt t}2 + \frac1{2\sqrt t} \geq 1$ for $t>0$.
If, moreover, $C^{-1/2}\widetilde C C^{-1/2} -I$ is trace class we can bound
\[
	d^2_\mathrm{TV}(\mu,\widetilde \mu) 
	\leq
	\frac 32
	\|C^{-1}\widetilde C - I\|_\mathrm{HS} + \frac 12 \|C^{-1/2}(m-\widetilde m)\|_{\mc H}
\]
where $\|A\|_\mathrm{HS} \coloneqq \sqrt{\tr(A^*A)}$ denotes the Hilbert--Schmidt norm of a Hilbert--Schmidt operator $A\colon \mc H \to \mc H$.
This bound is derived by Pinsker's inequality and means of \cite{DevroyeEtAl2019}.
Thus, using Theorem \ref{theo:TV} we can bound the total variation distance of posteriors resulting from Gaussian priors with different mean or covariance.
However, we remark that Gaussian priors on function spaces are often singular w.r.t.~each other.
For example, Gaussian priors with Mat\'ern covariance operator $C = C_{\alpha, \beta,\gamma} = \beta(I + \gamma^2 \Delta)^{-\alpha}$ \cite{LasanenEtAl2014,DunlopEtAl2017} are singular for different values of $\alpha > 0$ or $\beta>0$.
We refer to \cite{DunlopEtAl2017} for a further discussion and for a particular subclass of equivalent Gaussian priors with Mat\'ern covariance.
\end{rem} 
%==========================================================================================
%
% STABILITY WRT KLD
%
\section{Stability in Kullback--Leibler Divergence}
\label{sec:kld}
A common way to compare the relative information between two probability measures $\mu, \widetilde \mu \in \mc P(E)$ is to compute the \emph{Kullback--Leibler divergence} (KLD) between them, which in case of existence of $\frac{\d \mu}{\d \widetilde \mu}$ is
\[
	d_\mathrm{KL}(\mu \| \widetilde \mu)
	\coloneqq
	\int_{E} \log\left( \frac{\d \mu}{\d \widetilde \mu}(x)\right) \, \mu(\d x).
\] 
If $\mu$ is not absolutely continuous w.r.t.~$\widetilde \mu$, then $d_\mathrm{KL}(\mu \| \widetilde \mu) \coloneqq +\infty$.
The KLD is not a metric for probability measures due to the lack of symmetry and triangle inequality\footnote{There exists a metric for probability measures based on the KLD called the \emph{Jensen--Shannon distance} \cite{EndresSchindelin2003}.} but nonetheless an important quantity in information theory and optimal experimental design.
Moreover, the total variation and Hellinger distance can be bounded by the KLD, see \eqref{equ:PM_bounds}.
In particular, the well-known \emph{Pinkser's inequality} states that $2d^2_\mathrm{TV}(\mu, \widetilde \mu) \leq d_\mathrm{KL}(\mu \| \widetilde \mu)$, see \cite{GibbsSu2002}.

In the following, we study the stability of $\muN$ in terms of the KLD w.r.t.~perturbations of the log-likelihood and the prior.
Previous results in this direction were obtained by \cite{Latz2019, MarzoukXiu2009} stating a continuous dependence of $\muN$ on the data $y\in\bbR^n$ \cite{Latz2019} and a stability of $\muN$ w.r.t.~numerical approximations of the forward map $G\colon E\to\bbR^n$ \cite{MarzoukXiu2009} under suitable assumptions.

\begin{theo}\label{lem:KLD_Phi}
Let $\mu \in \mc P(E)$ and $\Phi, \widetilde \Phi\colon  E \to \bbR$ belong to $L^1_{\mu}(\bbR)$ with $\essinf_{\mu} \Phi = 0$.
Then, for the two probability measures $\muN, \muA \in \mc P(E)$ given in \eqref{equ:mu} and \eqref{equ:mu_tilde_1}, respectively,
\begin{equation}\label{bayes:equ:mu_Phi_cont}
	d_\mathrm{KL}(\muN \| \muA) \leq \frac{2\exp(- [\essinf_{\mu} \widetilde \Phi]_- )}{\min(Z,\widetilde Z)}\,\|\Phi - \widetilde \Phi\|_{L^1_{\mu}}.
\end{equation}
\end{theo}
\begin{proof}
We have $\frac{\d \muN}{\d \muA}(x)
	= 
	\frac{\d \muN}{\d\mu}(x)
	\ \frac{\d \mu}{\d \muA}(x)
	= \frac{\widetilde Z}{Z} \e^{\widetilde \Phi(x) - \Phi(x)}$
and, thus,
\[
	d_\mathrm{KL}(\muN \| \muA)
	\leq
	|\log(\widetilde Z) - \log(Z)| + \int_{E} |\widetilde \Phi(x) - \Phi(x)| \, \frac{\e^{-\Phi(x)}}{Z}\, \mu(\d x).
\]
We further obtain
\[
	\int_{E} |\widetilde \Phi(x) - \Phi(x)| \, \frac{\e^{-\Phi(x)}}{Z}
	       \, \mu(\d x)
	\leq
	\frac 1Z \int_{E} |\widetilde \Phi(x) - \Phi(x)| \, \mu(\d x)
	= \frac{\|\Phi-\widetilde \Phi\|_{L^1_{\mu}}}{Z}
\]
as well as due to $|\log t - \log s| \leq \frac1{\min(t,s)} |t-s|$ for $t,s>0$ and $|\e^{-t}-\e^{-s}|\leq \e^{-\min(t,s)} |t-s|$ for $t,s\in \bbR$ that
\[
	|\log(\widetilde Z) - \log(Z)|
	\leq
	\frac1{\min(Z,\widetilde Z)} |Z-\widetilde Z|
	\leq \frac{\exp(- [\essinf_{\mu} \widetilde \Phi]_- )}{\min(Z,\widetilde Z)} \|\Phi-\widetilde \Phi\|_{L^1_{\mu}},
\]
cf. the proof of Theorem \ref{theo:hell_Phi} for $|Z-\widetilde Z| \leq \e^{- [\essinf_{\mu} \widetilde \Phi]_- }\|\Phi-\widetilde \Phi\|_{L^1_{\mu}},$
Hence, we end up with
\[
	d_\mathrm{KL}(\muN \| \muA)
	\leq \frac{\exp(- [\essinf_{\mu} \widetilde \Phi]_- )}{\min(Z,\widetilde Z)} \|\Phi-\widetilde \Phi\|_{L^1_{\mu}} +  \frac{\|\Phi-\widetilde \Phi\|_{L^1_{\mu}}}{Z}
	\leq \frac{2\exp(- [\essinf_{\mu} \widetilde \Phi]_- )}{\min(Z,\widetilde Z)}\|\Phi-\widetilde \Phi\|_{L^1_{\mu}}.
\]
\end{proof}

An analogous proof to Theorem \ref{lem:KLD_Phi} also yields
\begin{equation}\label{bayes:equ:mu_Phi_cont}
	d_\mathrm{KL}(\muA \| \muN) \leq \frac{2\exp(- [\essinf_{\mu} \widetilde \Phi]_- )}{\min(Z,\widetilde Z)}\,\|\Phi - \widetilde \Phi\|_{L^1_{\mu}}.
\end{equation}

Concerning the following stability statement w.r.t.~perturbed priors $\widetilde \mu \in \mc P(E)$ we restrict ourselves to $\widetilde \mu$ which are equivalent to $\mu$.
Note that this assumption is sensible, since otherwise $d_\mathrm{KL}(\mu \| \widetilde \mu)$ or $d_\mathrm{KL}(\muN \| \muB)$ would be infinite.

\begin{theo}\label{theo:KLD_prior}
Let $\mu, \widetilde \mu \in \mc P(E)$ be equivalent and $\Phi\colon E \to [0,\infty)$ be measurable.
Then, for $\muN, \muB \in \mc P(E)$ given in \eqref{equ:mu} and \eqref{equ:mu_tilde_2}, respectively, we have
\[
	d_\mathrm{KL}(\muN \| \muB)
	\leq \frac{1}{\min(Z, \widetilde Z)}  \, \left(d_\mathrm{KL}(\mu \| \widetilde \mu) + d_\mathrm{KL}(\widetilde \mu \| \mu)\right),
	\qquad
	|Z - \widetilde Z| \leq \sqrt{2 d_\mathrm{KL}(\mu \| \widetilde \mu)}.	
\]
%as well as 
%\[
%	|Z - \widetilde Z| \leq \sqrt{2 d_\mathrm{KL}(\mu \| \widetilde \mu)}.
%\]
\end{theo}
\begin{proof}
Let $\rho(x) \coloneqq \frac{\d \mu}{\d \widetilde \mu}(x)$, then we have $\frac{\d \muN}{\d \muB}(x) = \frac{\widetilde Z}{Z} \rho(x)$ and obtain
\[
	d_\mathrm{KL}(\muN \| \muB)
	= \int_{E} \log \left( \frac {\widetilde Z}{Z} \rho(x) \right) \ \muN(\d x)
	= \int_{E} \log \left(\rho(x) \right) \ \muN(\d x) - \log\left(\frac{Z}{\widetilde Z}\right) 
\]
where the first term can be bounded as follows:
\[
	\int_{E} \log \left(\rho(x) \right) \ \muN(\d x)
	=  \int_{E} \log \left(\rho(x) \right) \ \frac{\exp(-\Phi(x))}{Z} \mu(\d x)
	\leq \frac 1Z d_\mathrm{KL}(\mu \| \widetilde \mu).
\]
Concerning the second term we first note that
\[
	\frac{Z}{\widetilde Z}
	=
	\frac1{\widetilde Z} \int_{E} \e^{-\Phi(x)} \ \mu(\d x)
	=
	\frac1{\widetilde Z} \int_{E} \e^{-\Phi(x)} \ \rho(x)\ \widetilde \mu(\d x)
	=
	\int_{E} \rho(x)\ \muB(\d x)
\]
and then apply Jensen's inequality for the convex function $t \mapsto -\log(t)$, $t>0$, to obtain
\[
	- \log\left(\frac{Z}{\widetilde Z}\right) 
	%& \leq &
	\leq
	\int_{E} -\log \left( \rho(x)\right) \ \muB(\d x)
	=
	\int_{E} \log \left( \frac1{\rho(x)}\right) \ \frac{\e^{-\Phi(x)}}{\widetilde Z} \widetilde \mu(\d x)
%	& \leq &
	\leq
	\frac1{\widetilde Z} \int_{E} \log \left( \frac{\d \widetilde \mu}{\d \mu}(x) \right) \  \widetilde \mu(\d x),
\]
where we used that $\frac{\d \widetilde \mu}{\d \mu}(x)  = \frac1{\rho(x)}$.
Hence, we end up with
\[
	d_\mathrm{KL}(\muN \| \muB)
	\leq \frac 1Z d_\mathrm{KL}(\mu \| \widetilde \mu) + \frac 1{\widetilde Z} d_\mathrm{KL}(\widetilde \mu \| \mu) 
\]
which yields the first statement.
The second statement is a direct implication of Pinsker's inequality and $|Z - \widetilde Z| \leq 2d_\mathrm{TV}(\mu, \widetilde \mu)$.
\end{proof}

Again, Theorem \ref{theo:KLD_prior} also implies a bound for the alternative KLD
\begin{equation}\label{bayes:equ:mu_Phi_cont}
	d_\mathrm{KL}(\muB \| \muN)
	\leq \frac{1}{\min(Z, \widetilde Z)}  \, \left(d_\mathrm{KL}(\mu \| \widetilde \mu) + d_\mathrm{KL}(\widetilde \mu \| \mu)\right).
\end{equation}

\begin{rem}[Kullback--Leibler divergence of Gaussian priors]\label{rem:Gaussian_KLD}
For Gaussian measures $\mu = N(m,C)$, $\widetilde \mu = N(\widetilde m, \widetilde C)$ on separable Hilbert spaces $\mc H$ there exists an exact formula for their KLD \cite{Pardo2006}:
Given $\mu$ and $\widetilde \mu$ are equivalent, and $C^{-1/2}\widetilde C C^{-1/2} -I$ is trace-class on $\mc H$ we have
\[
	d_\mathrm{KL}(\mu \| \widetilde \mu)
	=
	\frac 12\left(
	\tr(C^{-1}\widetilde C - I) + \|C^{-1/2}(m-\widetilde m)\|_{\mc H}^2 - \log \det(C^{-1}\widetilde C) \right).
\]
Again, this can be used combined with Theorem \ref{theo:KLD_prior} to bound the KLD of posterior measures resulting from (equivalent) Gaussian priors in terms of the pertubations in mean and covariance.
\end{rem}
%==========================================================================================
%
% STABILITY WRT WASSERSTEIN	
%
\section{Stability in Wasserstein Distance}
\label{sec:wasserstein}

In this section we focus on measuring perturbations of posterior and prior distributions in the Wasserstein distance.
The main advantage of this metric is that it does not rely on the absolute continuity of distributions.
Therefore, also for singular measures such as Dirac measures $\delta_x, \delta_{\widetilde x}\in \mc P(E)$ at $x \neq \widetilde x$ in $E$ the Wasserstein distance yields a sensible value which decays to $0$ as $d_E(x,\widetilde x)\to0$.
Besides that, the Wasserstein distance is based on the metric of the underlying space $E$ which allows some flexibility in the application by employing a suitable metric.
We introduce the following spaces of probability measures on a complete and separable metric space $(E,d_E)$ given a $q\geq 1$: 
\[
	  \mc P^q(E)
	\coloneqq
	\left\{
	\mu \in \mc P(E)\colon
	|\mu|_{\mc P^q} < \infty
	\right\},
	\qquad
	|\mu|_{\mc P^q}
	\coloneqq
	\inf_{x_0 \in E}
	\left(\int_{E} d_E^q(x, x_0) \ \mu(\d x)\right)^{1/q}.
\]
For measures $\mu, \widetilde \mu \in \mc P^q(E)$ we can now define the $q$-\emph{Wasserstein distance} by
\[
	\W_{q}(\mu, \widetilde \mu)
	\coloneqq
	\inf_{\pi \in \Pi(\mu,\widetilde \mu)} \left( \int_{E\times E} d_E^q(x,y) \ \pi(\d x \d y)\right)^{1/q},
\]
where $\Pi(\mu,\widetilde \mu)$ denotes the set of all \emph{couplings} $\pi \in \mc P(E\times E)$ of $\mu$ and $\widetilde \mu$, i.e., $\pi(A \times E) = \mu(A)$ and $\pi(E \times A) = \widetilde \mu(A)$ for each $A \in \mc E$.
We note that $(\mc P^q, \W_{p})$ is again a complete and separable metric space, see, e.g., \cite{Villani2009}.

We focus on the $1$-Wasserstein distance $\W_{1}$ subsequently. 
The advantage of this particular distance is its dual representation also known as \emph{Kantorovich--Rubinstein duality} \cite{Villani2009}:
\[
	\W_1(\mu, \widetilde \mu)	
	=
	\sup_{f\colon E \to \bbR,\ \Lip(f) \leq 1} \left| \int_{E} f(x) \ \mu(\d x) - \int_{E} f(x) \ \widetilde \mu(\d x) \right|,
\]
where $\Lip(f) \coloneqq \sup_{x\neq y \in E} \frac{|f(x)-f(y)|}{d_E(x,y)}$ denotes the global Lipschitz constant of $f$ w.r.t.~the metric $d_E$ on $E$.
Our first result considers stability in Wasserstein distance w.r.t.~perturbations of the log-likelihood function.

\begin{theo}\label{lem:Wasserstein_Phi}
Let $\mu \in \mc P^2(E)$ and assume $\Phi, \widetilde \Phi\colon  E \to \bbR$ belong to $L^2_{\mu}(\bbR)$ with $\essinf_{\mu} \Phi = 0$.
Then, for the two probability measures $\muN,\muA \in \mc P(E)$ given in \eqref{equ:mu} and \eqref{equ:mu_tilde_1}, respectively, we have
\begin{equation}\label{equ:Wasserstein_Phi}
	\W_1(\muN, \muA) 
	\leq \frac{\exp(- [\essinf_{\mu} \widetilde \Phi]_- )}{\widetilde Z}\left(|\mu_\Phi|_{\mc P^1} \| \Phi - \widetilde \Phi\|_{L^1_{\mu}} + |\mu|_{\mc P^2}\| \Phi - \widetilde \Phi\|_{L^2_{\mu}}\right).
\end{equation}
\end{theo}
\begin{proof}
Let $x_0\in E$ be arbitrary. 
We start with the dual representation
\[
	\W_1(\muN	, \muA)	
	=
	\sup_{\Lip(f) \leq 1,\ f(x_0)=0} \left| \int_{E} f(x) \ (\muN(\d x) - \muA(\d x)) \right|
\]
where we can take the supremum also w.l.o.g.~w.r.t.~all Lipschitz continuous functions $f\colon E \to \bbR$ with $\Lip(f)  = \sup_{x\neq y \in E} \frac{|f(x)-f(y)|}{d_E(x,y)} \leq 1$ and $f(x_0)=0$.
The latter two conditions imply $|f(x)| \leq d_E(x,x_0)$.
Furthermore, we have that 
\[
	\left|\int_{E} f(x) \ (\muN(\d x) - \muA(\d x))\right|
	= \left| \int_{E} f(x) \ \left( \frac{\e^{-\Phi(x)}}{Z} - \frac{\e^{-\widetilde \Phi(x)}}{\widetilde Z} \right) \mu(\d x)\right|
	\leq I_1(f) + I_2(f)
\]
where
\[
	I_1(f) \coloneqq \left|\frac1Z - \frac 1{\widetilde Z}\right| \left| \int_{E} f(x) \ \e^{-\Phi(x)}\mu(\d x) \right|,\qquad
	I_2(f) \coloneqq \left| \frac 1{\widetilde Z} \int_{E} f(x) \ \left( \e^{-\Phi(x)} - \e^{-\widetilde \Phi(x)} \right) \mu(\d x) \right|.
\]
We can bound $I_2$ as follows using $|\e^{-t}-\e^{-s}|\leq \e^{-\min(t,s)} |t-s|$ for $t,s\in \bbR$ and the Cauchy--Schwarz inequality:
\begin{eqnarray*}
	\sup_{\Lip(f) \leq 1,\ f(x_0)=0} I_2(f) 
	&\leq & \frac {\exp(- [\essinf_{\mu} \widetilde \Phi]_- )}{\widetilde Z} \int_{E} d_E(x,x_0) \ |\Phi(x) - \widetilde \Phi(x)| \, \mu(\d x)\\
	& \leq & \frac {\exp(- [\essinf_{\mu} \widetilde \Phi]_- )}{\widetilde Z}  \left(\int_{E} d^2_E(x,x_0) \mu(\d x)\right)^{1/2} \ \|\Phi - \widetilde \Phi\|_{L^2_{\mu}}.
\end{eqnarray*}
Moreover, due to $|\frac1Z - \frac 1{\widetilde Z}| = \frac {|\widetilde Z - Z|}{Z\widetilde Z}$ and $|Z - \widetilde Z| \leq \e^{- [\essinf_{\mu} \widetilde \Phi]_- } \|\Phi - \widetilde \Phi\|_{L^1_{\mu}}$, see Theorem \ref{theo:hell_Phi}, we have
\begin{eqnarray*}
	\sup_{\Lip(f) \leq 1,\ f(x_0)=0}  I_1(f)
	& \leq & \frac{\exp(- [\essinf_{\mu} \widetilde \Phi]_- )}{\widetilde Z}\|\Phi - \widetilde \Phi\|_{L^1_{\mu}}\ \sup_{\Lip(f) \leq 1,\ f(x_0)=0} \left| \int_{E} f(x) \ \muN(\d x) \right|\\
	& \leq & \frac{\exp(- [\essinf_{\mu} \widetilde \Phi]_- )}{\widetilde Z} \|\Phi - \widetilde \Phi\|_{L^1_{\mu}}\ \int_{E} d_E(x,x_0) \mu_\Phi(\d x).
\end{eqnarray*}
Since $x_0\in E$ was chosen arbitrarily we obtain the statement.
\end{proof}
If one prefers an estimate without $|\muN|_{\mc P^1}$, then we can bound $\W_1(\muN,\muA)$ also by
\begin{eqnarray*}
	\W_1(\muN,\muA) 
	& \leq & \frac{\e^{- [\essinf_{\mu} \widetilde \Phi]_- }|\mu|_{\mc P^1} \|\Phi - \widetilde \Phi\|_{L^1_{\mu}}}{Z\widetilde Z} + \frac{\e^{- [\essinf_{\mu} \widetilde \Phi]_- } |\mu|_{\mc P^2} \|\Phi - \widetilde \Phi\|_{L^2_{\mu}}}{\widetilde Z}\\
	& \leq & \frac{2\e^{- [\essinf_{\mu} \widetilde \Phi]_- }|\mu|_{\mc P^2}}{\min(Z,\widetilde Z)^2} \|\Phi - \widetilde \Phi\|_{L^2_{\mu}},
\end{eqnarray*}
where we used $|\mu_\Phi|_{\mc P^1} \leq \frac{1}{Z} |\mu|_{\mc P^1}$ in the first line and in the second line $|\mu|_{\mc P^1} \leq |\mu|_{\mc P^2}$ and $\|\Phi - \widetilde \Phi\|_{L^1_{\mu}}\leq \|\Phi - \widetilde \Phi\|_{L^2_{\mu}}$ due to Jensen's inequality as well as $\min(Z,\widetilde Z)\leq Z \leq 1$.

As outlined in Remark \ref{rem:data_G}, we can use Theorem \ref{lem:Wasserstein_Phi} to show a (local Lipschitz) continuous dependence of the posterior measure w.r.t.~the observed data $y\in\bbR^n$ in Wasserstein distance.
This is done in detail at the end of this section under similar conditions as for Hellinger well-posedness, cf. \cite{DashtiStuart2017,Latz2019}.

Studying the stability w.r.t.~the prior in Wasserstein distance is unfortunately more delicate than in the previous sections and the following result requires some restrictive assumptions which we discuss afterwards.

\begin{theo}\label{theo:Wasserstein_prior}
Let $E$ be bounded w.r.t.~the metric $d_E$, i.e.,
\[
	\sup_{x,y \in E} d_E(x,y) \leq D < \infty,
\]
and let $\e^{-\Phi}\colon E \to [0,1]$ be Lipschitz w.r.t.~$d_E$, i.e., $\Lip(\e^{-\Phi}) < \infty$.
Then, for the two probability measures $\muN,\muB \in \mc P(E)$ given in \eqref{equ:mu} and \eqref{equ:mu_tilde_2}, respectively, we have
\[
	\W_1(\muN, \muB)
	\leq 
	\frac{1 + D\Lip(\e^{-\Phi})}{\widetilde Z} \
	\left( 1 + \Lip(\e^{-\Phi}) \frac{|\mu|_{\mc P^1}}{Z} \right) \W_1(\mu,\widetilde \mu)
\]
and
\[
	|Z - \widetilde Z| \leq \Lip(\e^{-\Phi}) \ \W_1(\mu, \widetilde \mu).
\]
\end{theo}
\begin{proof}
Again, let $x_0\in E$ be arbitrary.
By the duality of $\W_1$ we have
\[
	\W_1(\muN, \muB)	
	=
	\sup_{\Lip(f) \leq 1,\ f(x_0)=0} \left| \int_{E} f(x) \ \e^{-\Phi(x)} \ \left( \frac{\mu(\d x)}{Z} - \frac{\widetilde \mu(\d x)}{\widetilde Z} \right)\right|.
\]
For any $f\colon E \to \bbR$ with $\Lip(f) \leq 1$ and $f(x_0)=0$ we get that $g(x) \coloneqq f(x)\e^{-\Phi(x)}$ satisfies $g(x_0) = 0$ as well as 
\begin{eqnarray*}
	|g(x) - g(y)|
	& \leq & |f(x)| \ |\e^{-\Phi(x)}- \e^{-\Phi(y)}| + |\e^{-\Phi(y)}| \ |f(x) - f(y)|\\
	& \leq & |f(x)| \ \Lip(\e^{-\Phi}) d_E(x,y) + d_E(x,y)\\
	&\leq & \left(1 + D\Lip(\e^{-\Phi}) \right) d_E(x,y)
\end{eqnarray*}
since $|f(x)| \leq |f(x_0)| + |f(x)-f(x_0)| \leq d_E(x,x_0) \leq D$.
Hence, we obtain  
\[
	\W_1(\muN, \muB)	
	\leq \left( 1 + D\Lip(\e^{-\Phi}) \right)
	\sup_{\Lip(g) \leq 1,\ g(x_0)=0} \left| \int_{E} g(x) \ \left( \frac{\mu(\d x)}{Z} - \frac{\widetilde \mu(\d x)}{\widetilde Z} \right)\right|
\]
and by the triangle inequality
\begin{eqnarray*}
	\frac{\W_1(\muN, \muB)}{1 + D\Lip(\e^{-\Phi})}
	& \leq  & \sup_{\Lip(g) \leq 1,\ g(x_0)=0} \left[ \left(\frac 1Z - \frac 1{\widetilde Z}\right) \left| \int_{E} g(x) \ \mu(\d x)\right|  
	%\right.
	%\\
	%& & \qquad \left. 
	+ \frac 1{\widetilde Z}  \left| \int_{E} g(x) \ \left(\mu(\d x) - \widetilde \mu(\d x)\right)\right| \right]\\
	& \leq & \frac{|Z -\widetilde Z|}{Z\widetilde Z} \int_{E} d_E(x,x_0)\ \mu(\d x) + \frac1{\widetilde Z} \W_1(\mu, \widetilde \mu).
\end{eqnarray*}
Since $x_0\in E$ was chosen arbitrarily  and
\[
	|Z-\widetilde Z| = \left| \int_{E} \e^{-\Phi(x)} (\mu(\d x) - \widetilde \mu(\d x)) \right|
	\leq
	\Lip(\e^{-\Phi})\W_1(\mu, \widetilde \mu),
\]
we obtain the statement.
\end{proof}
A slightly worse but maybe more convenient bound than the one given in Theorem \ref{theo:Wasserstein_prior} is
\[
	\W_1(\muN, \muB)
	\leq 
	\frac{\left(1 + D\Lip(\e^{-\Phi})\right)^2}{\min(Z,\widetilde Z)^2}  \W_1(\mu,\widetilde \mu),
\]
which is derived by using $Z\leq 1$ and $|\mu|_{\mc P^1} \leq D$ due to the boundedness of $E$.

The assumption on the boundedness of $d_E$ on $E$ is not that restrictive, since we can always consider a bounded version $\widetilde d_E(x,y) \coloneqq \min(D, d_E(x,y))$, $D >0$, of a metric $d_E$ on $E$ and, thereby, obtain a bounded metric space $(E,\widetilde d_E)$.
However, a crucial restriction in Theorem \ref{theo:Wasserstein_prior} is the Lipschitz condition $\Lip(\e^{-\Phi}) < \infty$ w.r.t.~a bounded metric on $E$. 
For example, for Euclidean spaces $E = \bbR^n$ equipped with the bounded metric $d_E(x,y) \coloneqq \min(D, |x-y|)$, $D>0$, and a sufficiently smooth $\Phi \in C^1(\bbR^n;[0,\infty))$ the condition $\Lip(\e^{-\Phi}) < \infty$ would require that
\[
	\sup_{x\in E} \|\nabla \e^{-\Phi(x)} \| = \sup_{x\in E}\frac{\|\nabla \Phi(x)\|}{\e^{\Phi(x)}} < \infty,
\]
where $\nabla$ denotes the gradient w.r.t.~the usual Euclidean norm $|\cdot|$ on $E=\bbR^n$.
This condition fails to hold, for instance, for functions $\Phi\colon \bbR^n \to [0,\infty)$ which are bounded but have growing derivatives such as $\Phi(x) = 1+\sin(\exp(x))$, $x\in\bbR$. 

However, we present the following result stating the continuous dependence of the posterior on the prior in Wasserstein distance.
Here, we consider the general $q$-Wasserstein distance, since the proof does not rely on the particular Kantorovich--Rubinstein duality of the $1$-Wasserstein distance.
\begin{lem}
Let $q>0$ and consider a $\mu \in \mc P^q(E)$ and a sequence of $\widetilde \mu^{(k)}\in \mc P^q(E)$, $k\in\bbN$, with corresponding $\muN$ as in \eqref{equ:mu} and
\[
	\muB^{(k)}(\d x) \coloneqq \frac{\e^{- \Phi(x)}}{\widetilde Z_k} \ \widetilde\mu^{(k)}(\d x), \qquad
	\widetilde Z_k \coloneqq \int_E\e^{- \Phi(x)}\ \widetilde\mu^{(k)}(\d x),
\]
given a measurable $\Phi\colon E\to [0,\infty)$. 
If $\Phi$ is continous, then
\[
	\W_q(\mu, \widetilde \mu^{(k)}) \to 0
	\quad \mathrm{ implies } \quad
	\W_q(\muN, \muB^{(k)}) \to 0.
\]
\end{lem}
\begin{proof}
We exploit the equivalence of convergence in $q$-Wasserstein distance and weak convergence, see \cite{Villani2009}: For $\nu,\nu^{(k)}\in \mc P^q(E)$, $k\in\bbN$, the statement $\lim_{k\to\infty} \W_q(\nu, \nu^{(k)}) \to 0$ is equivalent to
\[
	\nu^{(k)} \rightharpoonup\nu
	\quad \mathrm{and} \quad
	\lim_{k\to\infty} \int_E d(x,x_0)^q\ \nu^{(k)}(\d x) = \int_E d(x,x_0)^q\ \nu(\d x),
\]
where $x_0 \in E$ is arbitrary and $\rightharpoonup$ denotes weak convergence of measures, i.e., $\nu^{(k)} \rightharpoonup\nu$ means $\int f(x)\ \nu^{(k)}(\d x) \to \int f(x)\ \nu(\d x)$ for any bounded, continuous $f\colon E\to\bbR$.
Since $\W_q(\mu, \widetilde \mu^{(k)}) \to 0$, we know that for any such $f$ we have
\[
	\lim_{k\to \infty} \int_E f(x)\ \widetilde\mu^{(k)}(\d x) = \int_E f(x)\ \mu(\d x),
\] 
which implies that
\begin{eqnarray*}
	\lim_{k\to \infty} \int_E f(x)\ \muB^{(k)}(\d x) 
	& = &
	\frac{\lim_{k\to \infty}\int_E f(x)\ \e^{-\Phi(x)}\ \widetilde\mu^{(k)}(\d x) }{\lim_{k\to \infty} \int_E \e^{-\Phi(x)}\ \widetilde\mu^{(k)}(\d x)}
	=
	\frac{\int_E f(x)\ \e^{-\Phi(x)}\ \mu(\d x) }{\int_E \e^{-\Phi(x)}\ \mu(\d x)}\\
	&=&
	\int_E f(x)\ \muN(\d x),
\end{eqnarray*}
since $\e^{-\Phi}$ is continuous and bounded by assumption.
Thus, we have $\muB^{(k)} \rightharpoonup\muN$.
Moreover, we use that
\[
	\lim_{k\to \infty} \int_E d(x,x_0)^q\ \widetilde\mu^{(k)}(\d x) = \int_E d(x,x_0)^q\ \mu(\d x).
\] 
is equivalent to
\[
	\lim_{k\to \infty} \int_E f(x)\ \widetilde\mu^{(k)}(\d x) = \int_E f(x)\ \mu(\d x)
\] 
for any continuous $f\colon E\to\bbR$ with $|f(x)|\leq C(1+d(x,x_0)^q)$, $C\in\bbR$, see \cite[Definition 6.8]{Villani2009}.
Since $\e^{-\Phi}$ is continuous and bounded by one, we therefore obtain
\[
	\lim_{k\to \infty} \int_E d(x,x_0)^q\ \e^{-\Phi(x)}\ \widetilde\mu^{(k)}(\d x)
	=
	\int_E d(x,x_0)^q\ \e^{-\Phi(x)}\ \mu(\d x)
\]
which yields by the same arguments as above that
\[
	\lim_{k\to \infty} \int_E d(x,x_0)^q\ \muB^{(k)}(\d x)
	=
	\int_E d(x,x_0)^q\ \muN(\d x).
\]
Hence, the statement is proven.
\end{proof}

We close the discussion on Wasserstein stability with a few remarks on the results we have obtained.

\begin{rem}[Wasserstein distance of Gaussian priors]
The $2$-Wasserstein distance of Gaussian measures $\mu = N(m, C)$, $\widetilde \mu = N(\widetilde m, \widetilde C)$ on a separable Hilbert space $\mc H$ is given by
\[
	\W_2(\mu, \widetilde \mu)
	=
	\sqrt{\|m-\widetilde m\|^2_{\mc H} + \tr(C) + \tr(\widetilde C) - 2\tr\left( \sqrt{C^{1/2} \widetilde C C^{1/2}} \right) },
\]
see \cite[Theorem 3.5]{Gelbrich1990}.
This provides a bound for the $1$-Wasserstein distance of Gaussian priors, since $\W_1(\mu, \widetilde \mu) \leq \W_2(\mu, \widetilde \mu)$ due to Jensen's inequality.
Besides that we have for $\mu = N(m,C)$ that $\mu \in \mc P^2(\mc H)$ with $|\mu|_{\mc P^2} = \sqrt{\tr(C)}$. 
We highlight, that $\W_1(\mu, \widetilde \mu)$ or $\W_2(\mu, \widetilde \mu)$ does not depend on the equivalence of Gaussian priors $\mu, \widetilde \mu$.
\end{rem}

\begin{rem}[Increasing sensitivity]\label{rem:Sens_Wasserstein}
The bounds established in Theorem \ref{lem:Wasserstein_Phi} and \ref{theo:Wasserstein_prior} suggest also for the Wasserstein distance an increasing sensitivity of the posterior to perturbations of the log-likelihood or prior as the posterior becomes increasingly concentrated. 
In Remark \ref{rem:BvM} we have outlined why such an increasing sensitivity is quite natural in the topology induced by the total variation or Hellinger distance.
The same reasoning applies when pertubations are measured by the Kullback--Leibler divergence, since the KLD also relies on the equivalence of (perturbed posterior and prior) measures.
We now argue that this increasing sensitivity is also natural for the Wasserstein distance.
To this end, we consider a sequence of increasingly concentrated posterior measures $\muN^{(k)}(\d x) \coloneqq Z_k^{-1} \e^{- k\Phi(x)} \ \mu(\d x)$ for $k\in\bbN$ with $Z_k \coloneqq \int_E \e^{- k\Phi(x)} \ \mu(\d x)$.
Let $S \coloneqq \supp(\mu)$ and assume that $x_\star \coloneqq \argmin_{x \in S} \Phi(x)$ exists and is unique.
Then, under mild assumptions, $\muN^{(k)}$ converges weakly to $\delta_{x_\star}$, see \cite{Hwang1980}.
Given a perturbed prior $\widetilde \mu$ we set $\muB^{(k)}(\d x) \coloneqq \widetilde Z_k^{-1} \e^{- k\Phi(x)} \ \widetilde\mu(\d x)$ with $\widetilde Z_k \coloneqq \int_E \e^{- k\Phi(x)} \ \widetilde\mu(\d x)$.
If $\widetilde x_\star \coloneqq \argmin_{x \in \widetilde S} \Phi(x)$ exists and is unique, where $\widetilde S \coloneqq \supp(\widetilde\mu)$, then $\muB^{(k)}$ converges weakly to $\delta_{\widetilde x_\star}$ under mild assumptions \cite{Hwang1980}.
Thus, in order to have non-exploding local Lipschitz constants w.r.t.~the Wasserstein distance of the mappings $\mu \mapsto \muN^{(k)}$ as $k\to\infty$, we require that there exists a radius $r>0$ and a constant $C<\infty$ such that
\[
		\lim_{k\to\infty} \frac{\W_1(\muN^{(k)}, \muB^{(k)})}{\W_1(\mu, \widetilde\mu)} 
		\leq C
		\qquad
		\forall 
		\widetilde \mu\in\mc P^1(E)\colon \W_1(\mu, \widetilde\mu) \leq r.
\]
In the following, we assume that the metric $d_E$ of the complete and separable space $(E,d_E)$ is bounded.
This yields, given the weak convergence of $\muN^{(k)}$ to $\delta_{x_\star}$ and of $\muB^{(k)}$ to $\delta_{\widetilde x_\star}$, that
\[
		\lim_{k\to\infty} \frac{\W_1(\muN^{(k)}, \muB^{(k)})}{\W_1(\mu, \widetilde\mu)} 
		=
		\frac{\W_1(\delta_{x_\star}, \delta_{\widetilde x_\star}) }{\W_1(\mu, \widetilde\mu)} 
		=
		\frac{d_E(x_\star, \widetilde x_\star)}{\W_1(\mu, \widetilde\mu)},
\]
see \cite{Villani2009}.
Next, we construct a sequence of perturbed priors $\widetilde \mu^{(\epsilon)} \in \mc P^1(E)$ with $\W_1(\mu, \widetilde\mu^{(\epsilon)}) <\epsilon$, $\epsilon >0$, for which the ratio $d_E(x_\star,  \widetilde x^{(\epsilon)}_\star)/\W_1(\mu, \widetilde\mu^{(\epsilon)})$ deteriorates to infinity as $\epsilon \to 0$.
To this end, we consider a ball of radius $\epsilon>0$ around $x_\star$, i.e., $B_\epsilon(x_\star) \coloneqq \{x\in E\colon d_E(x,x_\star)\leq \epsilon\}$, and set for an arbitrarily chosen $x_\epsilon \in \partial B_\epsilon(x_\star)$
\[
	\widetilde \mu^{(\epsilon)}(A) \coloneqq \mu(A) - \mu(A \cap B_\epsilon(x_\star)) + \mu(B_\epsilon(x_\star))\ \delta_{x_\epsilon}(A),
	\qquad
	A \in \mc E,
\]
i.e., outside the ball $B_\epsilon(x_\star)$ the measure $\widetilde \mu^{(\epsilon)}$ coincides with $\mu$ but all the probability mass $\mu(B_\epsilon(x_\star))$ inside the ball $B_\epsilon(x_\star)$ is now concentrated at the single point $x_\epsilon$.
Assuming that $\Phi$ is continuous and $\epsilon$ sufficiently small we have $\widetilde x^{(\epsilon)}_\star \coloneqq \argmin_{x \in \supp \widetilde \mu^{(\epsilon)}} \Phi(x) \in \partial B_\epsilon(x_\star)$.
Thus, $d_E(x_\star,  \widetilde x^{(\epsilon)}_\star) = \epsilon$.
On the other hand,
\[
	\W_1(\mu, \widetilde\mu^{(\epsilon)})
	=
	\int_{B_\epsilon(x_\star)} d_E(x_\epsilon, x) \ \mu(\d x)
	\leq
	2\epsilon \mu(B_\epsilon(x_\star)).
\]
Hence, for suitable non-finite spaces $E$ and non-atomic priors $\mu$ such that $\lim_{\epsilon\to0} \mu(B_\epsilon(x)) = 0$ for any $x\in E$ we have
\[
	\lim_{\epsilon \to 0}\frac{d_E(x_\star,  \widetilde x^{(\epsilon)}_\star)}{\W_1(\mu, \widetilde\mu^{(\epsilon)})}
	\geq
	\lim_{\epsilon \to 0}
	\frac{1}{2\mu(B_\epsilon(x_\star))}
	=
	\infty.
\]
This shows that also in the Wasserstein topology, the posterior depends increasingly sensitively on perturbations of the prior as the likelihood becomes more informative.
A similar reasoning can be employed to show also the increasing sensitivity w.r.t.~perturbations of the log-likelihood measured in $L^p_\mu$-norms. 
\end{rem}

\paragraph{Wasserstein Well-posedness.} 
In the following we show how the stability result in Theorem \ref{lem:Wasserstein_Phi} can be used to establish well-posedness of Bayesian inverse problems (BIP) in Wasserstein distance.
The well-posedness of BIP w.r.t.~Hellinger distance, including a local Lipschitz-continuous dependence on the data, has been studied in a number of works \cite{DashtiStuart2017, Hosseini2017, HosseiniNigam2017, Latz2019, Stuart2010, Sullivan2017}.
Recently, well-posedness has been extended to the Kullback--Leibler divergence and Wasserstein distance in \cite{Latz2019}, but stating only a continuous dependence on the data.
We prove a local Lipschitz dependence on the observed data in Wasserstein distance based on Theorem \ref{lem:Wasserstein_Phi}.

We briefly recall the Bayesian setting from Section \ref{sec:pre}: Given a prior $\mu \in \mc P(E)$ for the unknown and an observed realization $y\in\bbR^n$ of $Y \coloneqq G(X) + \varepsilon$, $X\sim \mu$ and $\varepsilon \sim \nu_\varepsilon$, the resulting posterior $\muN$ is of the form \eqref{equ:mu} with $\Phi(x) \coloneqq \ell(y - G(x))$.
Here, $\ell$ denotes the negative log-density of the $\nu_\varepsilon(\d \epsilon) \propto \exp(- \ell(\epsilon)) \d \epsilon$.
We now show a local Lipschitz continuous dependence of the posterior $\muN$ on the data $y\in\bbR^n$ in Wasserstein distance---given the same basic assumption on $\Phi$ or $\ell$, respectively, required in \cite{DashtiStuart2017, Stuart2010} and slightly modified in \cite{Hosseini2017, HosseiniNigam2017, Sullivan2017} for the Hellinger distance.

\begin{cor} \label{cor:Wasserstein_wellposed_Stuart}
Let $\muN \in \mc P(E)$ be given as in \eqref{equ:mu} with $\Phi(x) = \ell(y - G(x))$ and assume that $G \colon E \to \bbR^n$ and $\ell \colon \bbR^n \to [0,\infty)$ are measurable. 
Furthermore, we assume there exists a monotonic and non-decreasing function $M\colon [0,\infty)\times \bbR \to [0,\infty)$ such that for any $y,\widetilde y\in\bbR^n$ with $|y|,|\widetilde y| \leq r < \infty$, $r>0$, we have
\[
	\left|\ell(y-G(x)) - \ell(\widetilde y-G(x))\right|
	\leq
	M(r, \|x\|) \ |y-\widetilde y|
	\qquad
	\forall x \in E,
\]
as well as $M(r,\|\cdot\|) \in L^2_{\mu}(\bbR)$.
If there exists a bounded set $A\subset E$ with $\mu(A)>0$, then there exists for any $r>0$ a constant $C_r <\infty$ such that for each $|y|,|\widetilde y| \leq r$ we have
\[
	\W_1(\muN,\muA) \leq C_r |y-\widetilde y|
\]
where $\muA$ is as in \eqref{equ:mu_tilde_1} with $\widetilde \Phi(x) = \ell(\widetilde y - G(x))$.
\end{cor}
\begin{proof}
By construction, we have $\essinf_{\mu} \widetilde \Phi \geq 0$ and obtain by means of Theorem \ref{lem:Wasserstein_Phi}
\[	 
	\W_1(\muN, \muA) 
	\leq 
	\frac{2|\mu|_{\mc P^2}}{\min(Z,\widetilde Z)^2} \|\Phi - \widetilde \Phi\|_{L^2_{\mu}}
	\leq
	\frac{2|\mu|_{\mc P^2}}{\min(Z,\widetilde Z)^2} \|M(r,\|\cdot\|) \|_{L^2_{\mu}}\ |y-\widetilde y|,
\]
where the last inequality followed by our assumption.
It remains to bound $\min(Z,\widetilde Z)$ uniformly (w.r.t.~$|y|,|\widetilde y|$) from below. 
By using the assumption again we obtain
\[
	\max\left(\ell(y-G(x)), \ell(\widetilde y-G(x))\right)  \leq \ell(0-G(x)) + rM(r,\|x\|).
\]
For the bounded set $A$ we define $R_A \coloneqq \sup_{x\in A} M(r,\|x\|) < \infty$ which then yields % due to $\mu(A)>0$
\begin{eqnarray*}
	\min(Z,\widetilde Z) 
	& \geq & \int_{E} \exp(- \ell(0-G(x)) - rM(r,\|x\|) )\ \mu(\d x)\\
	& \geq & \int_{A} \exp(- \ell(0-G(x)) - rM(r,\|x\|) )\ \mu(\d x)\\
	& \geq & \exp(-rR_A ) \int_{A} \exp(- \ell(0-G(x)) \ \mu(\d x) > 0,
\end{eqnarray*}
due to $\mu(A)>0$. This concludes the proof.
\end{proof}
By similar arguments and appropriate assumptions, cf.  \cite[Section 4.2]{DashtiStuart2017}, local Lipschitz continuity in Wasserstein distance for converging approximation $G_h$ of $G$, i.e., $\lim_{h\to0} G_h(x) = G(x)$ for all $x\in E$, can be shown.
%%==========================================================================================
%
% LITERATURE REVIEW	
%
\section{Discussion of Related Literature}
\label{sec:literature}

Besides the rather recent well-posedness studies of Bayesian inverse problems, the idea of a robust Bayesian analysis and the question about the sensitivity of the posterior w.r.t.~the prior measure (or the likelihood function) have a long history in Bayesian statistics. 
Some of the early references are \cite{Dempster1975, Good1950, Huber1973} and convenient overviews of many existing approaches and (positive) results are given in \cite{Berger1990,Berger1994,InsuaRuggeri2000}.

A common approach in robust Bayesian analysis is to consider a class of possible and sensible priors $\Gamma \subset \mc P(E)$, or likelihood functions, and to study and bound the range of a functional of interest $f\colon \mc P(E)\to \bbR$ over the set of resulting posterior measures, i.e., to estimate $\inf_{\mu \in \Gamma} f(\muN)$ and $\sup_{\mu \in \Gamma} f(\muN)$.
These bounds can then be used for robust decision making accounting for a variation of the prior, or likelihood.
Typical functionals of interest are, for instance, probabilities of certain events, e.g., $f(\mu) = \mu(A)$, $A\in\mc E$, the (Fr\'echet) mean of $\mu$ or the covariance of $\mu$ if $E$ is a linear space.
There exist several common types of classes of priors with corresponding bounds on the range of various functionals $f$.
We refer to the literature above and focus only on a particular, appealing type of class---the \emph{$\epsilon$-contamination class}---later on.

Moreover, in the described setting of robust Bayesian analysis also a notion of non-robustness or instability of Bayesian inference has been established, called the \emph{dilation phenomenon} \cite{WassermanSeidenfeld1994}.
This occurs if
\[
	\inf_{\mu \in \Gamma} f(\muN) \leq \inf_{\mu \in \Gamma} f(\mu)
	\leq
	\sup_{\mu \in \Gamma} f(\mu) \leq \sup_{\mu \in \Gamma} f(\muN)
\]
with one of the outer inequalities being strict.
Thus, dilation means that the posterior range of $f$ is larger than the prior range of $f$ over the class $\Gamma$.
Recently, an extreme kind of dilation, called \emph{Bayesian brittleness}, was established in \cite{OwhadiScovel2016, OwhadiEtAl2015a, OwhadiEtAl2015b} w.r.t.~(a) arbitrarily small perturbations of the likelihood and (b) classes of priors $\Gamma_k\subset \mc P(E)$ specified only by $k\in\bbN$ moments or other functionals.

Another approach to robust Bayesian analysis, starting with \cite{DiaconisFreedman1986}, considers the Fr\'echet and G\^ateaux derivative of the posterior measure $\muN$ w.r.t.~perturbations of the prior measure $\mu + \rho$ where $\rho$ denotes a suitable signed measure of mass zero.
This leads to a derivative-based sensitivity analysis of Bayesian inference, see, e.g., \cite{DeyBirmiwal1994,GelfandDey1991, GustafsonWasserman1995}.
Already in these works, particularly \cite{DiaconisFreedman1986,GustafsonWasserman1995}, the increasing sensitivity of the posterior measure in case of an increasing amount of observational data was noticed.

In the following we discuss in more detail the relation of our stability results to the classical robust Bayesian analysis for $\epsilon$-contamination classes of prior measures as well as to the derivative-based sensitivity analysis of posterior measures, and, moreover, explain why our results do not contradict Bayesian brittleness.

\paragraph{Robustness for $\epsilon$-contamination classes.}
A commonly used class of admissible priors in robust Bayesian analysis are $\epsilon$-contamination classes:
Given a reference prior $\mu \in \mc P(E)$ and a set of suitable perturbing probability measures $\mc Q \subset \mc P(E)$, we consider the class
\[
	\Gamma_{\epsilon, \mc Q}(\mu)
	\coloneqq
	\{
		(1-\epsilon) \mu + \epsilon \nu \colon \nu \in \mc Q
	\}
	\subset \mc P(E),
	\qquad
	\epsilon >0.
\] 
Common choices for $\mc Q$ are simply $\mc Q = \mc P(E)$, all symmetric and unimodal distributions on $E$, or all distributions such that $(1-\epsilon) \mu + \epsilon \nu$ is unimodal if $\mu$ is.
The choice $\mc Q = \mc P(E)$ is, of course, the most conservative and comes closest to our setting.
For brevity we denote $\Gamma_{\epsilon}(\mu) \coloneqq \Gamma_{\epsilon, \mc P(E)}(\mu)$ in the following.
If we consider now balls $B^\mathrm{TV}_\epsilon$ of radius $\epsilon>0$ in $\mc P(E)$ w.r.t.~total variation distance $d_\mathrm{TV}$ we have 
\[
	\Gamma_{\epsilon}(\mu)
	\subset
	B^\mathrm{TV}_\epsilon(\mu)
	\coloneqq
	\{\widetilde \mu \in \mc P(E)\colon d_\mathrm{TV}(\mu, \widetilde \mu) \leq \epsilon\}, 	
\]
since $d_\mathrm{TV}\left( (1-\epsilon) \mu + \epsilon \nu, \mu\right) \leq \epsilon$.
However, the $\epsilon$-contamination class $\Gamma_{\epsilon}(\mu)$ is in general a strict subset of the ball $B^\mathrm{TV}_\epsilon(\mu)$, because $\supp \mu \subseteq \supp [(1-\epsilon) \mu + \epsilon \nu]$ whereas there exist probability measures $\widetilde \mu$ with $d_\mathrm{TV}(\mu, \widetilde \mu) \leq \epsilon$ but $\supp \mu \nsubseteq \supp \widetilde \mu$.
Thus, our prior stability results are, in general, w.r.t.~a larger class of perturbed prior measures than $\epsilon$-contamination classes.

Furthermore, we establish a local Lipschitz continuous dependence of the posterior measure on the prior w.r.t.~particular probability distances such as the total variation distance. 
This is, in general, a different concept than bounding the posterior range of functionals of interest.
Of course, for certain cases we can find relations.
For example, concerning probabilities, i.e., functionals $f_A(\mu) \coloneqq \mu(A)$ where $A\in\mc E$, a local Lipschitz continuity in terms of the total variation distance as in Theorem \ref{theo:TV} implies also bounds on the posterior range of $f_A$ over $\Gamma_\epsilon(\mu)$.
In particular, we obtain with the results of Section \ref{sec:hellinger} that for all $A \in \mc E$
\[
	\inf_{\widetilde \mu \in \Gamma_{\epsilon}(\mu)} \muB(A)
	\geq 
	\muN(A) - \frac{2\epsilon}Z,
	\qquad
	\sup_{\widetilde \mu \in \Gamma_{\epsilon}(\mu)} \muB(A)
	\leq 
	\muN(A) + \frac{2\epsilon}Z.
\]
However,  in \cite{Huber1973} we find explicit expressions for the range of posterior probabilities for an $A\in\mc E$ over the class $\Gamma_{\epsilon}(\mu)$:
\begin{eqnarray*}
	\inf_{\widetilde \mu \in \Gamma_{\epsilon}(\mu)}
	\muB(A)
	& = &
	\mu_\Phi(A)
	\left( 1 + \frac{\epsilon \sup_{x\notin A} \exp(-\Phi(x))}{(1-\epsilon)Z}\right)^{-1},\\
	\sup_{\widetilde \mu \in \Gamma_{\epsilon}(\mu)}
	\muB(A)
	& = &
	\frac{(1-\epsilon)Z\mu_\Phi(A) + \epsilon \sup_{x\in A} \exp(-\Phi(x))}{(1-\epsilon)Z + \epsilon \sup_{x\in A} \exp(-\Phi(x))}.
\end{eqnarray*}
On the other hand, these exact bounds do not allow the derivation of local Lipschitz continuity w.r.t.~the total variation distance on $\Gamma_{\epsilon}(\mu)$, because they do not imply a bound for $|\muB(A) - \muN(A)|$ by a constant times $d_\mathrm{TV}(\mu,\widetilde \mu)$. 
Nonetheless, these exact ranges can be used to study lower bounds for the total variation distance of perturbed posteriors:
\begin{eqnarray*}
	\sup_{\widetilde \mu \in B^\mathrm{TV}_\epsilon(\mu)} d_\mathrm{TV}(\muN, \muB)
	& \geq &
	\sup_{\widetilde \mu \in \Gamma_{\epsilon}(\mu)}
	d_\mathrm{TV}(\muN, \muB)
	=
	\sup_{A \in \mc E}\, \sup_{\widetilde \mu \in \Gamma_{\epsilon}(\mu)}
	|\muN(A) - \muB(A)|\\
	& = &
	\sup_{A \in \mc E} \max\left\{\muN(A) - \inf_{\widetilde \mu \in \Gamma_{\epsilon}(\mu)}
	\muB(A),\ \sup_{\widetilde \mu \in \Gamma_{\epsilon}(\mu)}
	\muB(A) -\muN(A)\right\}.
\end{eqnarray*}

\paragraph{Bayesian brittleness.}
In \cite{OwhadiScovel2016, OwhadiEtAl2015a, OwhadiEtAl2015b} the authors establish several results concerning an extreme instability of Bayesian inference w.r.t.~(a) small perturbations of the likelihood function and (b) w.r.t.~a class of priors specified only by finitely many ``generalized'' moments.
They call this instability \emph{brittleness} and state it w.r.t.~the posterior range of functionals\footnote{Actually, in \cite{OwhadiScovel2016, OwhadiEtAl2015a, OwhadiEtAl2015b} the functionals $f$ are functionals of the data distribution, i.e., $f\colon \mc P(\bbR^n)\to \bbR$.
However, in the parametric setting considered here, e.g., the distribution of the data or observable $Y$ given $x\in E$ is $N(G(x), \Sigma)$ for the Gaussian noise model, these functionals can be understood as functionals acting on $x \in E$, i.e., $f\colon E \to \bbR$.} $f\colon E \to \bbR$.

Their brittleness result concerning perturbed likelihood models is that for arbitrary small perturbations the resulting range of posterior expectations of $f$ is the same as the (essential) range of $f$ over the support of the prior $\mu$.
This result is no contradiction to the local Lipschitz stability shown in this paper.
The crucial difference between both results, brittleness and stability, is the way how perturbations of the likelihood are measured:
In \cite{OwhadiEtAl2015a, OwhadiEtAl2015b} the likelihood function $L$ is considered as a function of the parameter $x\in E$ and the data $y\in \bbR^n$---i.e,. $L(x,y) \propto \exp(- \Phi(x,y))$---and a perturbed likelihood $\widetilde L$---i.e., $\widetilde L(x,y) \propto \exp(-\widetilde \Phi(x,y))$---is considered close to $L$ if for all $x\in E$ the resulting data distribution on $\bbR^n$ with Lebesgue density $\widetilde L(x,\cdot)$ is close to the distribution with Lebesgue density $L(x,\cdot)$. 
For instance, employing the total variation distance for the induced data distributions on $\bbR^n$ we would consider $\widetilde L$ close to $L$ if $d_\mathrm{L}(L,\widetilde L) \coloneqq \sup_{x \in E} \|L(x,\cdot) - \widetilde L(x,\cdot)\|_{L^1}$ is small---here, the $L^1$-norm is taken w.r.t.~the Lebesgue measure on $\bbR^n$.
Thus, closeness of likelihood functions is considered in an average sense w.r.t.~the data $y$ but then uniformly w.r.t.~$x\in E$.
In this paper, on the other hand, we assume fixed data $y \in \bbR^n$ and consider the negative log-likelihoods $\Phi(\cdot) \coloneqq - \log L(\cdot, y)$ and $\widetilde \Phi(\cdot) \coloneqq -\log \widetilde L(\cdot, y)$ close to each other if $\|\Phi - \widetilde \Phi\|_{L^1_{\mu}}$ is small.
Thus, in our case closeness of log-likelihoods is considered in an average sense w.r.t.~the parameter $x\in E$ and for the fixed observed data $y\in\bbR^n$.
In Appendix \ref{sec:Brittle_App} we discuss in greater detail (i) why brittleness w.r.t.~the likelihood is natural if perturbations are measured by the distance $d_\mathrm{L}$ as above, and (ii) how stability can again be obtained if we employ the alternative distance $\widehat d_\mathrm{L}(L, \widetilde L) \coloneqq \sup_{y \in \bbR^n} \|L(\cdot, y) - \widetilde L(\cdot, y)\|_{L^1_{\mu}}$.
Note that the latter distance implies bounds on the perturbed marginal likelihood or evidence $\widetilde Z = \int_E \widetilde L(x,y)\ \mu(\d x)$ whereas the first distance $d_\mathrm{L}$ does not. 
This fact yields the difference between stability and brittleness, see Appendix \ref{sec:Brittle_App}.

The second brittleness result in \cite{OwhadiEtAl2015a, OwhadiEtAl2015b} is stated for classes of priors on $E$ defined only by a set of finitely many functionals\footnote{Again, in \cite{OwhadiEtAl2015a, OwhadiEtAl2015b} the functionals $\Psi_k$ are actually functionals of the data distribution associated with $x\in E$, i.e., $\Psi_k\colon \mc P(\bbR^n)\to \bbR$.} $\Psi_k\colon E\to\bbR$, $k=1,\ldots,K$.
In particular, given a measure $\nu_0 \in \mc P(\bbR^K)$ we consider the class $\Gamma \coloneqq \{\mu \in \mc P(E)\colon \Psi_*\mu = \nu_0\}$ of priors where $\Psi(x)\coloneqq(\Psi_1(x),\ldots,\Psi_K(x))$ and $\Psi_*\mu$ denotes the pushforward measure.
This construction accounts for the fact that in practice only finitely many information are available in order to derive or choose a prior measure.
In \cite{OwhadiEtAl2015a, OwhadiEtAl2015b} it is then shown under mild assumptions that the range of posterior expectations of an $f\colon E \to \bbR$ resulting from priors $\widetilde \mu \in \Gamma$ coincides with the range of $f$ on $E$.
Again, this is not a contradiction to the local Lipschitz stability w.r.t.~the prior established in this paper, since the class $\Gamma$ is, in general, quite different from balls $B_r(\mu) \subset \mc P(E)$ with radius $r>0$ around a reference prior $\mu$ in Hellinger or Wasserstein distance.

\paragraph{Derivative of the posterior and local sensitivity diagnostics.} 
Besides the rather global pertubation estimates derived in the  robust Bayesian analysis for, e.g., contamination classes of prior measures, 
several authors studied the local sensitivity of the posterior measure w.r.t.~the prior.
As a first result we mention the \emph{derivative} of the posterior $\mu_\Phi$ w.r.t.~the prior $\mu$ in the total variation topology introduced by \cite{DiaconisFreedman1986} as follows.
Let $T_\Phi \colon \mc P(E) \to \mc P(E)$ denote the map from prior $\mu$ to posterior $T_\Phi(\mu) \coloneqq \muN(\d x) = \frac 1Z \exp(-\Phi(x)) \, \mu(\d x)$.
In order to define the derivative of $T_\Phi$ we consider the set $\mc S_0(E)$ of signed measures $\rho\colon \mc E\to\bbR$ on $E$ with zero mass $\rho(E) = 0$ for modelling perturbations of probability measures, i.e., perturbed priors $\widetilde \mu = \mu + \rho$.
We introduce the set of all admissible perturbations\footnote{In \cite{DiaconisFreedman1986} the authors allow for any perturbation $\rho \in \mc S_0(E)$ extending the application of $T_\Phi$ also to signed measures.} $P_{\mu} \coloneqq \{\rho \in \mc S_0(E)\colon \mu + \rho \in \mc P(E)\}$ of a prior $\mu\in\mc P(E)$ and notice that $P_{\mu}$ is star-shaped with center $\rho_0 = 0$.
Then the derivative $\partial T_\Phi(\mu)$ of $T_\Phi$ at a prior $\mu\in \mc P(E)$ is defined as the linear map from $P_{\mu}$ to $\mc S_0(E)$ satisfying
\[
	\lim_{\|\rho\|_\mathrm{TV}\to0} \frac{\|T_\Phi(\mu + \rho) - T_\Phi(\mu) - \partial T_\Phi(\mu)\rho\|_\mathrm{TV}}{\|\rho\|_\mathrm{TV}}
	=0,
\]
where $ \|\rho\|_\mathrm{TV} \coloneqq \int_E \left| \frac{\d \rho}{\d \nu} \right| \ \d \nu$ denotes the total variation norm of a (signed) measure $\rho$ with $\rho \ll \nu$ for a $\sigma$-finite measure $\nu$.
In \cite[Theorem 4]{DiaconisFreedman1986} it is then shown that
\[
	\partial T_\Phi(\mu)\rho
	=
	\frac 1Z \e^{-\Phi} \left(\rho - \frac {\int_E \e^{-\Phi} \d \rho}{Z} \mu \right)
	\in \mc S_0(E).
\]
Moreover, \cite[Theorem 4]{DiaconisFreedman1986} states the following bounds for the norm of the derivative $\|\partial T_\Phi(\mu)\| \coloneqq \sup_{\|\rho\|_\mathrm{TV} = 1} \|\partial T_\Phi(\mu) \rho\|_\mathrm{TV}$:
\[
	\frac 1Z \sup_{x\in E \colon \mu(\{x\})=0} \e^{-\Phi(x)} \ \leq \ \|\partial T_\Phi(\mu)\| \ \leq\ \frac1Z \sup_{x\in E}\e^{-\Phi(x)},
\]
i.e., for non-atomic priors $\mu$ we have $\|\partial T_\Phi(\mu)\| = \frac1Z$ given our standing assumption $\inf_x \Phi(x) = 0$.
This already implies an increasing sensitivity of the posterior w.r.t.~perturbations of the prior for increasingly informative likelihoods, i.e., a decreasing normalization constant $Z$.

Based on the Fr\'echet derivative $\partial T_\Phi(\mu)$ at $\mu$ other authors studied the sensitivies of $T_\Phi$ w.r.t.~a given class of possible perturbations, see, e.g., \cite{DeyBirmiwal1994, GelfandDey1991,GustafsonWasserman1995}.
For instance, given an $\epsilon$-contamination class $\Gamma_{\epsilon, \mc Q}(\mu)$ as above the authors of \cite{GustafsonWasserman1995} study the sensitivity $s(\mu, \mc Q;\Phi) \coloneqq \sup_{\nu \in\mc Q} s(\mu, \nu; \Phi)$ with local sensitivies
\[
	s(\mu, \nu; \Phi)
	\coloneqq
	\lim_{\epsilon \to 0}
	\frac{d_\mathrm{TV}( T_\Phi(\mu), T_\Phi((1-\epsilon) \mu + \epsilon \nu))}{d_\mathrm{TV}(\mu, (1-\epsilon) \mu + \epsilon \nu)}.
\]
Since $(1-\epsilon) \mu + \epsilon \nu = \mu + \epsilon (\nu -\mu)$ and $d_\mathrm{TV}(\mu, (1-\epsilon) \mu + \epsilon \nu) = \epsilon \|\nu - \mu\|_\mathrm{TV}$, this local sensitivity coincides with the norm of the G\^ateaux derivative of $T_\Phi$ at $\mu$ in the direction $\rho = \nu-\mu \in \mc S_0(E)$, i.e., $s(\mu, \nu; \Phi) = \|T_\Phi(\mu)(\nu-\mu)\|_\mathrm{TV}$.
In \cite{GustafsonWasserman1995} the authors consider furthermore \emph{geometric} perturbations of the prior such as $\widetilde \mu(\d x) \propto \left(\frac{\d \nu}{\d \mu}\right)^\epsilon \mu(\d x)$, $\epsilon >0$, and local sensitivities based on divergences rather than total variation distance, see also \cite{DeyBirmiwal1994, GelfandDey1991} employing the Kullback--Leibler divergence.
Again, they derive an increasing sensitivity $s(\mu, \mc Q;\Phi)\to\infty$ for various classes $\mc Q$ as the likelihood $\e^{-\Phi}$ becomes more informative due to more observations in their case.
In particular, they derive explicit growth rates of $s(\mu, \mc Q;\Phi_N)$ w.r.t.~$N\in\bbN$ where $N$ denotes the number of i.i.d.~observations employed for Bayesian inference and $\Phi_N$ the corresponding log-likelihood.

These results on Fr\'echet or G\^ateaux derivatives w.r.t.~the prior measure are quite close to our approach establishing explicit bounds on the local Lipschitz constant. 
In particular, the constant $C_{\mu,\Phi}(r)$ in the corresponding result \eqref{equ:cont_mu} can be seen as an upper bound on the norm of the derivative $\|\partial T_\Phi(\widetilde \mu)\|$ for all perturbed priors $\widetilde \mu \in B_r(\mu)$ belonging to the $r$-ball around $\mu$ in $\mc P(E)$---cf. Theorem \ref{theo:TV} stating that $d_\mathrm{TV}(\muN, \muB) \leq \frac{2}{Z}\,d_\mathrm{TV}(\mu, \widetilde \mu)$. 
Compared to the studies in \cite{DeyBirmiwal1994, GelfandDey1991,GustafsonWasserman1995} we allow for arbitrary perturbed priors not restricted to (geometric) $\epsilon$-contamination classes and, moreover, we consider different topologies on $\mc P(E)$ induced by Hellinger distance, Kullback--Leibler divergence and Wasserstein distance.

%==========================================================================================
\paragraph{Acknowledgments.}
The author would like to thank Jonas Latz, Han Cheng Lie and Daniel Rudolf for valuable comments, as well as Claudia Schillings and Tim Sullivan for helpful discussions and for the encouragement to write this note.
Moreover, the author acknowledges the financial support by the DFG within the project 389483880.

%==========================================================================================
%
% APPENDIX
%

\appendix
\section{On Brittleness and Stability w.r.t.~Perturbed Likelihoods}
\label{sec:Brittle_App}

In this appendix we discuss in more detail the phenomenon of Bayesian brittleness for perturbed likelihoods as stated in \cite[Theorem 6.4]{OwhadiEtAl2015a}.
Moreover, we reveal the mathematical reason behind the brittleness and show how one can obtain stability by modifying the distance for the likelihood functions.

\paragraph{Setting.} We first recall the setting in \cite{OwhadiEtAl2015a,OwhadiEtAl2015b}. 
We assume a fixed prior measure $\mu \in \mc P(E)$ and for simplicity only consider the parametric case where the distribution of the observable data on $\bbR^n$ depends only on $x\in E$.
I.e., consider a prior distributed random variable $X\sim \mu$ on $E$ and an observable random variable $Y$ on $\bbR^n$ such that the conditional distribution of $Y$ given $X=x$ is given by $\nu_x \in \mc P(\bbR^n)$ with $\nu_x(\d y) = L(x,y) \d y$ for a positive Lebesgue density $L(x,\cdot) \colon \bbR^n \to (0,\infty)$.
Thus, $L(x,\cdot) \in L^1(\bbR^n)$ for all $x\in E$ and we suppose that $L\colon E\times \bbR^n \to (0,\infty)$ is jointly measurable.
Moreover, rather than observing a precise realization $y \in \bbR^n$ of $Y$ we suppose that we observe the event $Y \in B_\delta(y) \subset \bbR^n$, i.e., we account for a finite resolution of the data described by the radius $\delta>0$ of the ball $B_\delta(y) = \{y' \in \bbR^n\colon |y-y'|\leq \delta\}$.
Conditioning $X\sim \mu$ on the observation $Y \in B_\delta(y)$ yields a posterior probability measure on $E$ depending on $L$ which we denote by
\[	 
	\mu_L(\d x \ | \ B_\delta(y))
	\coloneqq
	\frac {\exp(- \Phi_L(x))}{Z_L} 
	\ \mu(\d x),\\
	\qquad
	\Phi_L(x)
	\coloneqq
	- \log \int_{B_\delta(y)} L(x,y')\ \d y',
\]
where $Z_L \coloneqq \int_E \exp(- \Phi_L(x))\ \mu(\d x)$.

\paragraph{Bayesian brittleness.}
Let us now consider a perturbed likelihood  model, namely, another jointly measurable $\widetilde L\colon E \times \bbR^n \to (0, \infty)$ such that $\int_{\bbR^n} \widetilde L(x,y) \ \d y = 1$ for all $x \in E$.
This model yields a perturbed posterior measure which we denote by
\[	 
	\mu_{\widetilde L}(\d x \ | \ B_\delta(y))
	\coloneqq
	\frac {\exp(- \Phi_{\widetilde L}(x))}{Z_{\widetilde L} }
	\ \mu(\d x),\\
	\qquad
	\Phi_{\widetilde L}(x)
	\coloneqq
	- \log \int_{B_\delta(y)} \widetilde L(x,y')\ \d y',
\]
and $Z_{\widetilde L} \coloneqq \int_E \exp(- \Phi_{\widetilde L}(x))\ \mu(\d x)$.
We can then ask for stability of the mapping $L\mapsto \mu_L$.
To this end, we measure the distance between the two likelihood models $L$, $\widetilde L$ by the following distance:
\[
	d_\mathrm{L}(L, \widetilde L)
	\coloneqq
	\sup_{x \in E} \|L(x,\cdot) - \widetilde L(x,\cdot)\|_{L^1}
	=
	2 \sup_{x \in E} d_\mathrm{TV}(\nu_x,  \widetilde \nu_x)
\]
where $\widetilde \nu_x \in \mc P(\bbR^n)$ denotes the probability measure on $\bbR^n$ induced by $\widetilde L(x, \cdot)$.
Although, this distance seems natural for comparing parametrized models for data distributions it leads to instability, or brittleness, as stated in \cite[Theorem 6.4]{OwhadiEtAl2015a}: Let $f\colon E \to \bbR$ be a measurable quantity of interest and consider the posterior expectation of $f$ which we simply denote by
\[
	\mu_L(f \ | \ B_\delta(y))
	\coloneqq
	\int_E f(x)\ \mu_L(\d x \ | \ B_\delta(y));
\]
then for each $\epsilon>0$ there exists a $\delta(\epsilon)>0$ such that
\[
	\sup_{\widetilde L\colon d_\mathrm{L}(L, \widetilde L) \leq \epsilon}
	\mu_{\widetilde L}(f \ | \ B_\delta(y))
	\geq
	\esssup_{\mu} f
	\qquad
	\forall 0 < \delta < \delta(\epsilon),\ \forall y\in \bbR^n,
\]
with an analogous statement for the infimum.
Thus, in other words, the range of all (perturbed) posterior expectations of $f$ resulting from all perturbed likelihood models $\widetilde L$ within an $\epsilon$-ball around $L$ w.r.t.~$d_\mathrm{L}$ covers the range of all (essential) prior values of $f$---as long as the observation is sufficiently accurate, i.e., $\delta < \delta(\epsilon)$.

\paragraph{An explanation for brittleness.}
We explain the Bayesian brittleness and the mathematical reason behind in terms of the total variation distance of the posterior measures:
\[	 
	d_\mathrm{TV}(\mu_L(\cdot \ | \ B_\delta(y)), \mu_{\widetilde L}(\cdot \ | \ B_\delta(y)) )
	= \frac 12\int_E \left|\frac 1{Z_{L} }
	\exp(- \Phi_{L}(x)) - \frac 1{Z_{\widetilde L} }
	\exp(- \Phi_{\widetilde L}(x)) \right| \mu(\d x).
\]
Similarly to Theorem \ref{theo:TV} we have
\[
	d_\mathrm{TV}(\mu_L(\cdot \ | \ B_\delta(y)), \mu_{\widetilde L}(\cdot \ | \ B_\delta(y)) )
	\leq
	\frac 1{Z_L} \int_E \left|	\exp(- \Phi_{L}(x)) - \exp(- \Phi_{\widetilde L}(x)) \right| \mu(\d x).
\]
Thus, using the definition of $\Phi_L$ and $\Phi_{\widetilde L}$, we can further bound
\begin{equation} \label{equ:brittle_bound}
	d_\mathrm{TV}(\mu_L(\cdot \ | \ B_\delta(y)), \mu_{\widetilde L}(\cdot \ | \ B_\delta(y)) )
	\leq \frac 1{Z_L} \int_E \int_{B_\delta(y)} \left| L(x,y') - \widetilde L(x,y')\right| \ \d y' \ \mu(\d x).
\end{equation}
Hence, for stability we need to control the $L^1$-difference of $|L(x,\cdot) - \widetilde L(x,\cdot)|$ over the observed event, the ball $B_\delta(y)$.
However, the bound $d_\mathrm{L}(L, \widetilde L)< \epsilon$ only implies that
\[
	\int_{B_\delta(y)} \left| L(x,y') - \widetilde L(x,y')\right| \ \d y'
	\leq
	\frac{\epsilon}{|B_\delta(y)|}
	\qquad
	\forall x\in E,
\]
where $|B_\delta(y)|$ denotes the Lebesgue measure of $B_\delta(y_0)\subset\bbR^n$.
Thus, for any $\epsilon$ we can take a sufficiently small $\delta$ and then $\epsilon/|B_\delta(y_0)|$ becomes arbitrarily large.
Now, of course, these are just discussions about controlling upper bounds for the total variation distance between the posteriors, but it should be clear that we can easily construct sufficiently ``bad'' perturbed likelihoods $\widetilde L$ with $d_\mathrm{L}(L,\widetilde L)<\epsilon$ but $d_\mathrm{TV}(\mu_L(\cdot \ | \ B_\delta(y)), \mu_{\widetilde L}(\cdot \ | \ B_\delta(y)) ) \approx 1$, see, for instance, the illustrative example in \cite[pp.~574--575]{OwhadiEtAl2015b}.

\paragraph{Obtaining stability.} 
The above estimate \eqref{equ:brittle_bound} suggests that stability w.r.t.~perturbed likelihoods can only be obtained in a distance for likelihoods $L$ and $\widetilde L$ which allows to control $|L(x,y) - \widetilde L(x,y)|$ uniformly w.r.t.~$y$.
Thus, if we employ the following alternative distance given the fixed prior $\mu$
\[
	\widehat d_\mathrm{L}(L, \widetilde L)
	\coloneqq
	\sup_{y \in \bbR^n} \|L(\cdot,y) - \widetilde L(\cdot,y)\|_{L^1_{\mu}},
\]
then we get by Fubini's theorem that
\begin{eqnarray*}
	d_\mathrm{TV}(\mu_L(\cdot \ | \ B_\delta(y)), \mu_{\widetilde L}(\cdot \ | \ B_\delta(y)) )
	&\leq & \frac 1{Z_L} \int_E \int_{B_\delta(y)} \left| L(x,y') - \widetilde L(x,y')\right| \ \d y' \ \mu(\d x)\\
	&\leq & \frac 1{Z_L} \widehat d_\mathrm{L}(L, \widetilde L),
\end{eqnarray*}
i.e., a local Lipschitz stability.
We remark that using the distance $\widehat d_\mathrm{L}$ implies that we bound the range of the possible likelihoods for the observed event.
As discussed before such a control is crucial for a stability w.r.t.~perturbed likelihood models. 

\section{Hellinger Distance of Gaussian Measures on Separable Hilbert Spaces}\label{app:Hell}
We provide a proof of the explicit expressions for the Hellinger distance of Gaussian measures on a separable Hilbert space $\mc H$ stated in Remark \ref{rem:Hell_Gauss}, since this is missing so far in the literature to the best of our knowledge.

\paragraph{Different means, same covariance.}
We start with proving that if $\widetilde m - m \in \rg C^{1/2}$, then
\[
	d^2_\mathrm{H}(N(m, C), N(\widetilde m, C))
	=
	2
	-
	2\exp\left(-\frac 18 \|C^{-1/2} (m-\widetilde m) \|_{\mc H}^2\right).
\]
To this end, we require the well-known Cameron--Martin formula for the density of $\widetilde \mu \coloneqq N(\widetilde m, C)$ w.r.t.~$\mu \coloneqq N(m, C)$.
This density is, given that $h \coloneqq \widetilde m - m \in \rg C^{1/2}$,
\[
	\frac{\d \widetilde \mu}{\d \mu}(x)
	=
	\exp\left(-\frac 12 \|C^{-1/2} h\|^2_{\mc H} + \langle C^{-1} h, x - m \rangle  \right),
	\qquad
	x\in\mc H,
\]
where $\langle C^{-1} h, \cdot - m \rangle\colon \mc H \to \bbR$ is well-defined as a random variable in $L^2_\mu(\bbR)$, see, e.g., \cite[Chapter 1]{DaPrato2004}.
We  then use that
\begin{equation}\label{equ:Hell_alt}
	d^2_\mathrm{H}(\mu, \widetilde \mu)
	=
	2 - 2 \int_{\mc H} \sqrt{\frac{\d \widetilde \mu}{\d \mu}(x)}\ \mu(\d x)
	\qquad
	\mathrm{ since } \qquad
	\widetilde \mu \ll \mu,
\end{equation}
which can be verified easily, and that for any $x' \in \mc H$ and $\mu = N(m,C)$
\[
	\int_{\mc H} \exp\left( \langle C^{-1/2} x', x - m \rangle \right) \ \mu(\d x)
	=
	\exp\left(\frac 12 \|x'\|_{\mc H}^2 \right),
\]
see \cite[Proposition 1.2.7]{DaPrato2004}, in order to derive that for $\mu = N(m, C)$, $\widetilde \mu = N(\widetilde m, C)$ with $h=\widetilde m - m \in \rg C^{1/2}$
\begin{eqnarray*}
	d^2_\mathrm{H}(\mu, \widetilde \mu)
	& = &
	2 - 2 \exp\left(-\frac 14 \|C^{-1/2} h\|^2_{\mc H} \right)
	\int_{\mc H} \exp\left(\frac 12 \langle C^{-1} h, x - m \rangle \right)\ \mu(\d x)\\
	& = &
	2 - 2 \exp\left(-\frac 14 \|C^{-1/2} h\|^2_{\mc H}\right)  \exp\left(\frac 18 \|C^{-1/2} h\|_{\mc H}^2 \right)\\
	& = &
	2
	-
	2\exp\left(-\frac 18 \|C^{-1/2} (m-\widetilde m) \|_{\mc H}^2\right).
\end{eqnarray*}

\paragraph{Same mean, different covariances.}
We now show that, for $\rg C^{1/2} = \rg \widetilde C^{1/2}$, $T\coloneqq C^{-1/2}\widetilde C C^{-1/2}$ being positive definite and $T - I$ being Hilbert--Schmidt on $\mc H$, we have
\[
	d^2_\mathrm{H}(N(m, C), N(m, \widetilde C))
	=
	2
	-
	2 \left[ \det\left(\frac 12 \sqrt T + \frac 12 \sqrt{T^{-1}}\right)\right]^{-1/2}.
\]
W.l.o.g.~we assume $m=0$ in the following and use \cite[Theorem 3.3]{Kuo1975} which states that for $\mu \coloneqq N(0,C)$ and $\widetilde \mu \coloneqq N(0,\widetilde C)$ and given the assumptions above, we have
\[
	\frac{\d \widetilde \mu}{\d \mu}
	(\psi(\xi))
	=
	\rho(\xi)
	\coloneqq
	\prod_{k=1}^\infty 
	\frac{1}{\sqrt{t_k}} \exp\left(\frac{t_k-1}{2 t_k} \xi_k^2\right),
	\qquad
	\xi = (\xi_1,\xi_2,\ldots)\in \bbR^\bbN,
\]
where the $t_k>0$, $k\in\bbN$, denote the eigenvalues of $T$ and the measurable mapping $\psi\colon \bbR^\bbN \to \mc H$ is specified in the proof of \cite[Theorem 3.3]{Kuo1975}.
We do not require the explicit definition of $\psi$, only the following relation which is also stated in the proof of \cite[Theorem 3.3]{Kuo1975}: With $\nu \coloneqq \bigotimes_{k=1}^\infty N(0,1)$ we have $\mu = \psi_* \nu$, i.e,. $\mu = N(0,C)$ is the pushforward of the product measure $\nu$ under the mapping $\psi$, see \cite{Kuo1975} for details.
We use these facts in combination with \eqref{equ:Hell_alt} to obtain that for $\mu \coloneqq N(0,C)$ and $\widetilde \mu \coloneqq N(0,\widetilde C)$
\begin{eqnarray*}
	d^2_\mathrm{H}(\mu, \widetilde \mu)
	& = &
	2 - 2 \int_{\mc H} \sqrt{	\rho( \psi^{-1}(x) ) }\ \psi_*\nu(\d x)
	=
	2 - 2 \int_{\mc H} \sqrt{	\rho( \xi ) }\ \nu(\d \xi)\\
	& = &
	2 - 2 \prod_{k=1}^\infty \int_\bbR \frac{1}{\sqrt[4]{t_k}}\ \exp\left(\frac{t_k-1}{4 t_k} \xi_k^2\right)\ \exp\left(- \frac 12 \xi_k^2 \right) \ \frac{\d \xi_k}{\sqrt{2\pi}}.
\end{eqnarray*}
A straightforward calculation yields
\[	 
	\int_\bbR\ \exp\left(\frac{t_k-1}{4 t_k} \xi_k^2\right)\ \exp\left(- \frac 12 \xi_k^2 \right) \ \d \xi_k
	=
	\int_\bbR\ \exp\left( -\frac 12 \frac{1+t_k}{2t_k} \xi_k^2\right) \ \d \xi_k
	=
	\sqrt{2\pi} \sqrt{\frac{2t_k}{1+t_k}} 
\]
and, thus,
\[	 
	d^2_\mathrm{H}(\mu, \widetilde \mu)
	=
	2 - 2 \prod_{k=1}^\infty \sqrt{ \frac{2\sqrt{t_k}}{1+t_k} }
	=
	2 - 2 \left[ \prod_{k=1}^\infty \frac{1+t_k}{2\sqrt{t_k}} \right]^{-1/2}
	=
	2 - 2 \left[ \prod_{k=1}^\infty \left( \frac{\sqrt t_k}{2} + \frac{1}{2\sqrt{t_k}}
	 \right)\right]^{-1/2},
\]
where we assumed for the moment that the infinite products converge.
Note, that the infinite product on the right-hand side coincides with $\det\left(\frac 12 \sqrt T + \frac 12 \sqrt{T^{-1}}\right)$ given that this (Fredholm) determinant is finite, i.e., given that $I - \left(\frac 12 \sqrt T + \frac 12 \sqrt{T^{-1}} \right)$ 
%\[
%	I - \left(\frac 12 \sqrt T + \frac 12 \sqrt{T^{-1}} \right)
%\]
is a trace-class operator.
Thus, if we can show that
\[
	\sum_{k=1}^\infty \left(1 - \frac{\sqrt t_k}{2} + \frac{1}{2\sqrt{t_k}}\right)
	=
	\sum_{k=1}^\infty \left(1 - \frac{1+t_k}{2\sqrt{t_k}}\right)
	 < \infty,
\]
then the above formula for $d^2_\mathrm{H}(N(m, C), N(m, \widetilde C))$ is verified.
We define the function $f(t) \coloneqq \frac{1+t}{2\sqrt{t}}$ for $t>0$ and compute its first and second derivative $f'(t) = \frac{ t^{1/2} - t^{-1/2}}{4t}$ and $f''(t) = \frac{3t^{-1/2} - t^{1/2} }{8t^2}$, respectively.
We notice that $f(1) = 1$ and $f'(1) = 0$, hence, 
\[
	\left|1 - \frac{1+t_k}{2\sqrt{t_k}}\right|
	=
	\left|f(1) - f(t_k)\right|
	\leq
	\max_{t \in [1,t_k]}
	|f''(t)|
	\ |1-t_k|^2.
\]
Moreover, we have that $t_k -1 \to 0$ as $k\to\infty$, since $T-I$ is Hilbert--Schmidt on $\mc H$.
Thus, there exists a $k_0 \in \bbN$ such that $|1-t_k| \leq \frac 12$ for $k \geq k_0$.
We obtain by setting $c \coloneqq \max_{t \in [\frac 12, \frac 32]} |f''(t)|  < \infty$ that
\[
	\left|1 - \frac{1+t_k}{2\sqrt{t_k}}\right|
	\leq
	c	
	\ |1-t_k|^2
	\qquad
	\forall k \geq k_0,
\]
which yields, since $T-I$ is Hilbert--Schmidt, that
\[
	\sum_{k=1}^\infty \left(1 - \frac{1+t_k}{2\sqrt{t_k}}\right)
	\leq
	\sum_{k=1}^{k_0} \left(1 - \frac{1+t_k}{2\sqrt{t_k}}\right)
	+
	c \sum_{k=k_0}^{\infty} (t_k-1)^2
	 < \infty.
\]

%==========================================================================================
%
% References
%
%\section*{References}
%\bibliographystyle{iopart-num}
\bibliographystyle{unsrt}
\bibliography{literature}

\end{document}